\documentclass[letterpaper,10pt]{article}

\usepackage[square,numbers]{natbib}

\usepackage[utf8]{inputenc}   
\usepackage[T1]{fontenc}  

\usepackage{pgf,tikz}
\usetikzlibrary{arrows}
\usepackage{multicol}
\usepackage{amsthm}
\usepackage{amsmath}
\usepackage{amssymb}
\usepackage{amsfonts}
\usepackage{stmaryrd}
\usepackage{mathabx}
\usepackage{mathtools}
\usepackage{latexsym}
\usepackage{color}
\usepackage{graphics,graphicx,graphpap}
\usepackage{multirow}
\usepackage{rotating}
\usepackage[new]{old-arrows}
\usepackage[all]{xy}
\usepackage[letterpaper]{geometry}
\usepackage{subfigure}
\usepackage{hyperref}
\usepackage{orcidlink}
\usepackage{authblk}
\theoremstyle{plain}

\usepackage[nottoc]{tocbibind}

\DeclareMathAlphabet{\mathpzc}{OT1}{pzc}{m}{it}
\newtheorem{theorem}{Theorem}
\newtheorem{prop}[theorem]{Proposition}
\newtheorem{cor}[theorem]{Corollary}
\newtheorem{lem}[theorem]{Lemma}

\newtheorem{que}[theorem]{Question}

\newcommand*{\email}[1]{%
    \normalsize\href{mailto:#1}{#1}\par
    }
\title{On independence complexes of graph products}
\author{Andr\'es Carnero Bravo\;\orcidlink{0009-0001-0221-5908}}
\affil{Centro de Ciencias Matemáticas, UNAM\\ \email{carnero@matmor.unam.mx}}

\begin{document}
\maketitle
\begin{abstract}
We study the independence complexes of graph products where at least one factor is a path. We also analyze the complexes of their induced subgraphs. We determine the homotopy 
type of the independence complex of the graphs $P_n\times P_m$, $P_n\boxtimes P_2$, $P_n\boxtimes P_3$ and $P_n\boxtimes P_4$. We also focus in the independence complexes of the induced subgraphs of 
$P_n\times P_3$, $P_n\boxtimes P_2$, $P_n\boxtimes P_3$, $P_n\boxtimes P_4$ and some lexicographic products $G\circ H$.
\end{abstract}
\tableofcontents
\textbf{Acknowledgments.} This work was supported by UNAM Posdoctoral Program (POSDOC).
\section{Introduction}
The independence complex is one of the most studied graph complexes in topological combinatorics,  it is define as the set of independent sets of a graph.
This complex has many applications in other branches in mathematics, for example it has been used 
to study the extreme Khovanov (co)homology of link diagrams (see \citep{Gonz_lez_Meneses_2017,marithaniacircindcompl,MR4741186}). Another way to think the independence complex is to define it as 
the clique complex of the complement graph, \textit{i.e.} the simplices are the complete graphs of the complement graph; in a finite metric space the Vietoris-Rips complexes are clique complexes 
(see for example \citep{adamaviet,nervcompcirarc,bendersky2023connectivity}) and this complexes are 
used in topological data analysis and geometric group theory (see for example \citep{toppersgeo}). 

In general is not easy to determine the homotopy type of the independence complex. Even for highly symmetric graphs the
topology of this complex can be really hard to understand, for example  
the chessboard complex-- which can be seen as the independence complex of the cartesian product of complete graphs-- is not well understood and has torsion for many cases 
(see \citep{Jonsson_2010,Shareshian_2007}). 
For this reason much of the work is concentrated in determining its homotopy type for specific families of graphs 
(see \citep{adamsplit,Bousquet_M_lou_2007,Braun_2017,Iriye_2012,Jonsson_2009,kozlovdire}) and graph products of particular graphs (see for example 
\citealp{indcomplcatprod,poljoin,matsushitan4n5,matsushitan6,MR4477852}).
Here we will focus in study the independence complexes of graph products in which one factor is a path, along with those of their induced subgraphs.
Our main focus will be the complexes of $P_n\boxtimes P_m$, but we will do some calculations for categorical products and lexicographic products.

\section{Preliminaries}
A simplicial complex $K$ is a family of subsets of a finite set $V(K)$, the vertices of the complex, such that if $\tau \subseteq \sigma$ and $\sigma\in K$, then $\tau\in K$. We will not 
distinguish between a simplicial complex and its geometric realization.

We are only interested in simple graphs. For a graph $G$, we denote its vertex set as $V(G)$ and its edge set as $E(G)$. The
\textit{order} of a graph $G$ is the non-negative integer number $|V(G)|$. For a vertex $v$, $N_G(v)=\{u\in V(G):\;uv\in E(G)\}$ is its \textit{open neighborhood} 
and $N_G[v]=N_G(G)\cup\{v\}$ its \textit{closed neighborhood}. The \textit{degree} of a vertex will be denoted $d_G(v)=|N_G(v)|$.

Given a graph $G$, its \textit{complement graph} is the graph $G^c$ with vertex set $V(G)$ and edge set $E(G^c)=\{\{u,v\}:\;\;\{u,v\}\notin E(G)\}$.

Taking $\underline{n}=\{1,\dots,n\}$, we define the \textit{path} on $n$ vertices as the graph $P_n$ with vertex set $\underline{n}$ and edge set 
$E(P_n)=\{ij:\;|i-j|=1\}$; we define the \textit{cycle} on $n\geq3$ vertices as the graph $C_n$ with vertex set $\underline{n}$ and edge set 
$E(C_n)=\{ij:\;|i-j|=1\}\cup\{1n\}$. $K_n$ is the \textit{complete graph} with vertex set $\underline{n}$ and edge set $\{\{i,j\}:\; i\neq j\}$. 

For two vertex disjoint graphs $G$ and $H$ their \textit{disjoint union} is the graph $G+H$ with vertex set $V(G)\cup V(H)$ and edge set $E(G)\cup E(H)$.
Given two graphs $G$ and $H$, we will consider four different graphs over the vertex set $V(G)\times V(H)$:
\begin{itemize}
    \item The \textit{categorical product} $G\times H$, with edge set 
    $$\{\{(u_1,v_1),(u_2,v_2)\}:\;\{u_1,u_2\}\in E(G)\mbox{ and }\{v_1,v_2\}\in E(H)\}.$$
    \item The \textit{cartesian product} $G\oblong H$ with edge set 
    $$\{\{(u_1,v),(u_2,v)\}:\;\{u_1,u_2\}\in E(G)\}\cup\{\{(u,v_1),(u,v_2)\}:\;\{v_1,v_2\}\in E(H)\}.$$
    \item  The \textit{strong product} $G\boxtimes H$ with edge set $E(G\times H)\cup E(G\oblong H)$.
    \item  The \textit{lexicographic product} $G\circ H$ with edge set
    $$\{\{(u,v_1),(u,v_2)\}:\;\{v_1,v_2\}\in E(H)\}\cup\{\{(u_1,v_1),(u_2,v_2)\}:\;\{u_1,u_2\}\in E(G)\}.$$
    
\end{itemize}
For all the graph definitions not stated here we follow \citep{graphsanddigraphs}.

A vertex set $\sigma\subseteq V(G)$ is \textit{independent} if no two vertices in $\sigma$ are adjacent. The \textit{independence complex} of $G$ is the simplicial complex 
$$I(G)=\{\sigma\subseteq V(G):\; \sigma \mbox{ is independent}\}.$$

\begin{lem}\label{lemmvert}
Let $u,v$ be different vertices of a graph $G$:
\begin{enumerate}
    \item \citep{engstrom09} If $N_G(u)\subseteq N_G(v)$, then $I(G)\simeq I(G-v)$.
    \item \citep{marithaniacircindcompl} If $N_G[u]\subseteq N_G[v]$, then $I(G)\simeq I(G-v)\vee\Sigma I(G-N_G[v])$.
\end{enumerate}
\end{lem}

\begin{lem}\citep{poljoin}\label{lemstarclsindgrp}
Let $G$ be a graph with $v$ a vertex of degree $2$ such that  $N_G(v)=\{v_1,v_2\}$ is an independent set. 
We take $H$ the graph obtain from $G$ by 
deleting the vertices in $N_G(v_1)\cap N_G(v_2)$ from $G$ and adding the edges $u_1u_2$ for all $u_1\in N_G(v_1)-N_G(v_2)$ and 
$u_2\in N_G(v_2)-N_G(v_1)$. Then 
$$I(G-N_G(v_1))\cup I(G-N_G(v_2))=I(H),$$
$$I(G-N_G(v))\cap\left(I(G-N_G(v_1))\cup I(G-N_G(v_2))\right)=I(H-v_1-v_2).$$
\end{lem}

For a vertex $v$ such that $N_G(v)$ is an independent set, we define its \textit{star cluster} as the subcomplex of $I(G)$ given by
$$SC(v)=\bigcup_{u\in N(v)}I(G-N_G(u)).$$

\begin{theorem}\label{barmak}\citep{barmak}
Let $G$ be a graph and let $v$ be a non-isolated vertex of $G$ which is contained in no triangle. Then
$$I(G)\simeq\Sigma(I(G-N_G(v))\cap SC(v)).$$
\end{theorem}

\begin{theorem}\label{theostrclstrdgr2}
Let $G$ be a graph with $v$ a vertex of degree $2$ such that  $N_G(v)=\{v_1,v_2\}$ is an independent set. Take $H$ as in 
Lemma \ref{lemstarclsindgrp}, then 
$$I(G)\simeq\Sigma I(H-v_1-v_2).$$
\end{theorem}
\begin{proof}
The result follows form Lemma \ref{lemstarclsindgrp} and Theorem \ref{barmak}.
\end{proof}

\begin{theorem}\citep{csorba}\label{edgesubd}
Let $G$ be a graph and $uv$ an edge. Take $H$ the graph obtained from $G$ by replacing the edge $uv$ by a path of length $4$. Then
$I(H)\simeq\Sigma I(G)$.
\end{theorem}

\begin{theorem}\label{kozlovpaths}\citep{kozlovdire}
\[I(P_n)\simeq\left\lbrace
\begin{array}{cc}
    \mathbb{S}^{k-1} & \mbox{if } n=3k  \\
    \ast & \mbox{if } n=3k+1\\
    \mathbb{S}^{k} & \mbox{if } n=3k+2   
\end{array}\right.,\;\;
I(C_n)\simeq\left\lbrace
\begin{array}{cc}
    \displaystyle\bigvee_2\mathbb{S}^{k-1} & \mbox{if } n=3k  \\
    \mathbb{S}^{k-1} & \mbox{if } n=3k+1 \mbox{ or } n=3k-1
\end{array}\right..\]
\end{theorem}

Before we conclude this section we define two families of graphs. We say a graph $G$ has property $\mathcal{P}$ if either
$I(G)$ is contractible, it has the homotopy type of a wedge of spheres or it has the homotopy type of a disjoint union of wedges of spheres. 
We say a graph $G$ has property $\mathcal{SP}$ if $I(G)$ is either contractible or it has the homotopy type of a wedge of spheres. 
From this, we define the following two families:
$$\mathcal{W}=\{G: \mbox{ any induced subgraph of } G \mbox{ has the property }\mathcal{P}\},$$
$$\mathcal{SW}=\{G: \mbox{ any induced subgraph of } G \mbox{ has the property }\mathcal{SP}\}.$$
For example, by Theorem \ref{kozlovpaths}, all the paths and cycles are in $\mathcal{SW}$. By the definition of the independence complex is 
clear that complete graphs are also in $\mathcal{SW}$. It is known that the independence complex of a forest is either contractible or 
has the homotopy type of a sphere \citep[see Corollary 6.1]{Ehrenborg_2006}, thus all forests are also in $\mathcal{SW}$. Lastly all graphs 
without induced cycles of length divisible by $3$ have independence complex homotopy equivalent to a sphere or contractible \citep{kimternary}, 
thus these graphs are also in $\mathcal{SW}$.

\begin{prop}\label{propdsjotunion}
If $G_1,\dots,G_k$ are in $\mathcal{W}$ or in $\mathcal{SW}$, then $G_1+\cdots+G_k$ are in $\mathcal{W}$ or in $\mathcal{SW}$ respectively. 
\end{prop} 
\begin{proof}
We take $H$ an induced subgraph of $G_1+\cdots+G_k$. Let $A=\{i_1,\dots,i_r\}$ the subset of $\underline{k}$ such that $H$ and $G_{i_j}$ have at least a common vertex for any $i_j$ in $A$. Then, for each 
$i_j$ in $A$ there is $H_{i_j}$ an induced subgraph of $G_{i_j}$ such that 
$$H=H_{i_1}+\dots+H_{i_r}.$$
Thus $I(H)=I(G_{i_1})*\dots*I(H_{i_r})$.
\end{proof}

\section{Independence complexes of categorical products}
In this section we give a decomposition formula for the homotopy type of the independence complex of $G\times P_n$
in terms of joins of the complex $I(G\times P_2)$. With this we calculate the homotopy type of $I(P_n\times P_m)$.
We begin giving a Lemma that was made as a commentary in \citep{indcomplcatprod} after the Proposition 3.1 of that article.
\begin{lem}\label{lemneigcatprod}
Let $G$ and $H$ be graphs. Assume there are vertices $u,v$ in $H$ such that $N_H(v)\subseteq N_H(u)$. Then 
$$I(G\times H)\simeq I(G\times(H-u)).$$
\end{lem}
\begin{proof}
For any vertex $w$ of $G$, we have that  
$$N_{G\times H}((w,v))=N_{G}(w)\times N_H(v)\subseteq N_{G\times H}((w,u))=N_G(w)\times N_H(u).$$
By Lemma \ref{lemmvert} we can delete all the vertices $(w,u)$ without changing the homotopy type.
\end{proof}
As immediate consequence we get the following corollary.  
\begin{cor}
For any graph $G$, we have that
$$I(G\times C_4)\simeq I(G\times P_2).$$
\end{cor}

\begin{theorem}\label{gtimesp2}
Let $G$ be a graph, then
$$I(G\times P_n)\simeq\left\lbrace
\begin{array}{cc}
 I(G\times P_2)^{*r} & \mbox{ if } n=3r\\
  *  & \mbox{ if } n=3r+1\\
  I(G\times P_2)^{*r+1} & \mbox{ if } n=3r+2
\end{array}\right..$$
\end{theorem}
\begin{proof}
The proof is by induction on $n$.
For $n=1,2$ this is clear. For $n\geq3$, By Lemma \ref{lemneigcatprod} we have that 
$$I(G\times P_n)\simeq I(G\times(P_n-3))\cong I(G\times P_2)*I(G\times P_{n-3}).$$
Thus $I(G\times P_3)\simeq I(G\times P_2)$. The rest follows by induction.
\end{proof}
If a graph $G$ is bipartite, then is easy to see that $G\times P_2\cong2G$. From this we get the following two corollaries.
\begin{cor}\label{catbip}
Let $G$ be a bipartite graph, then
$$I(G\times P_n)\simeq\left\lbrace
\begin{array}{cc}
 I(G)^{*2r} & \mbox{ if } n=3r\\
  *  & \mbox{ if } n=3r+1\\
  I(G)^{*2r+2} & \mbox{ if } n=3r+2
\end{array}\right..$$
\end{cor}

\begin{cor}
Let $G$ be a bipartite graph in $\mathcal{W}$. Then $G\times P_2$ is in $\mathcal{W}$.
\end{cor}

\begin{que}
If $G$ is a bipartite graph in $\mathcal{W}$, Does $G\times P_n$ is in $\mathcal{W}$ for all $n$?
\end{que}

Now we can easily determine the homotopy type of $I(P_n\times P_m)$.
\begin{prop}
    $$I(P_n\times P_m)\simeq\left\lbrace\begin{array}{cc}
    \displaystyle\mathbb{S}^{2kr-1} & \mbox{if } n=3k \mbox{ and } m=3r\\
    \displaystyle\ast  & \mbox{ if } n=3k+1 \mbox{ or } m=3r+1 \\
    \displaystyle\mathbb{S}^{2(k+1)r-1} & \mbox{if } n=3k+2 \mbox{ and } m=3r\\
    \displaystyle\mathbb{S}^{2(r+1)k-1} & \mbox{if } n=3k \mbox{ and } m=3r+2\\
    \displaystyle\mathbb{S}^{2(k+1)(r+1)-1} & \mbox{if } n=3k+2 \mbox{ and } m=3r+2
    \end{array}\right..$$
\end{prop}
\begin{proof}
The proposition follows from Theorem \ref{kozlovpaths} and Corollary \ref{catbip}.
\end{proof}

In \citep[Proposition 3.4.]{indcomplcatprod} it was proved that 
$$I(C_m\times P_2)\simeq\left\lbrace 
\begin{array}{cc}
\displaystyle\bigvee_{4}\mathbb{S}^{4k-1} & \mbox{if } m=6k\\
\mathbb{S}^{4k} & \mbox{if } m=6k+1 \\
\mathbb{S}^{4k+1} & \mbox{if } m=6k+r \;\mbox{with } r\in\{2,4\}\\
\mathbb{S}^{4k+1}\vee\mathbb{S}^{4k+1} & \mbox{if } m=6k+3 \\
\mathbb{S}^{4k+2} & \mbox{if } m=6k+5
\end{array}\right..$$
From this formula and Theorem \ref{gtimesp2} we get the following corollary.
\begin{cor}
$$I(C_m\times P_n)\simeq\left\lbrace\begin{array}{cc}
\ast   & \mbox{ if } n=3r+1\\
\displaystyle\bigvee_{4^{r}}\mathbb{S}^{4kr-1} & \mbox{if } m=6k \mbox{ and } n=3r\\
\mathbb{S}^{4kr+r-1} & \mbox{if } m=6k+1 \mbox{ and } n=3r\\
\mathbb{S}^{4kr+2r-1} & \mbox{if } m=6k+l \;\mbox{with } l\in\{2,4\}, \mbox{ and } n=3r\\
\displaystyle\bigvee_{2^{r}}\mathbb{S}^{4kr+2r-1} & \mbox{if } m=6k+3 \mbox{ and } n=3r\\
\mathbb{S}^{4kr+3r-1} & \mbox{if } m=6k+5 \mbox{ and } n=3r\\
\displaystyle\bigvee_{4^{r+1}}\mathbb{S}^{4k(r+1)-1} & \mbox{if } m=6k \mbox{ and } n=3r+2\\
\mathbb{S}^{4k(r+1)+r} & \mbox{if } m=6k+1 \mbox{ and } n=3r+2\\
\mathbb{S}^{4k(r+1)+2r+1} & \mbox{if } m=6k+l \;\mbox{with } l\in\{2,4\}, \mbox{ and } n=3r+2\\
\displaystyle\bigvee_{4^{r+1}}\mathbb{S}^{4k(r+1)+2r+1} & \mbox{if } m=6k+3 \mbox{ and } n=3r+2\\
\mathbb{S}^{4k(r+1)+3r+2} & \mbox{if } m=6k+5 \mbox{ and } n=3r+2
\end{array}\right.$$
\end{cor}

From what we have done, if $G$ is a  cycle or a bipartite graph such that $G^c$ is connected, then $I(G\times P_2)$ is simply connected. This 
can be generalized for any graph $G$ such that $G^c$ is connected.

\begin{theorem}
If $G$ is a a graph such that $I(G)$ is connected, then $I(G\times P_n)$ is simply connected. Even more, for $r\geq1$
$$\mathrm{conn}(I(G\times P_{3r}))\geq2r-1\;\;\mbox{and}\;\;\mathrm{conn}(I(G\times P_{3r+2}))\geq2r+1.$$
\end{theorem}
\begin{proof}
For $n=1$, this is clear. If we show the theorem is true for $n=2$, then, by Theorem \ref{gtimesp2}, it is true for any $n\geq3$.
Assume $n=2$ and let $T$ be a spanning tree of $G^c$. For $i=1,2$, we define $T_i$ as the copy of $T$ in $\Delta^{V(G)\times\{i\}}$ and we take $x$ a vertex 
in $V(G)$. We take $$\mathcal{T}=T_1+T_2+\{(x,1),(x,2)\}$$
a spanning tree of $\left(G\times P_2\right)^c$. Now we take the free group $H_{\mathcal{T}}$ with generators $\{uv:\; \{u,v\}\in E(\mathcal{T})\}$ and  
the relations  
\begin{itemize}
    \item $uv=vu=1$ for all the edges of $\mathcal{T}$.
    \item $(ty)(yz)=tz$ if $\{t,y,z\}$ is an independent set.
\end{itemize}
Then $\pi_1\left(I(G\times P_2),(x,1)\right)\cong H_{\mathcal{T}}$
(see \citep[Theorem 7.34]{rotmantop}). We will see that $H_{\mathcal{T}}$ is trivial. Take $e$ an edge of $(G\times P_2)^c$: 
\begin{itemize}
    \item If $e=\{u,v\}$ is an edge of $\mathcal{T}$, then $uv=vu=1$ by definition of the group.
    \item If $e=\{(w_1,1),(w_2,1)\}$ and it is not in $T_1$, we take $y_0y_1\cdots y_r$ the only path form $w_1$ to $w_2$ in $T_1$. Then $(w_1,1)(w_2,1)=(y_1,1)(w_2,1)$ because 
    $\{(w_1,1),(w_2,1),(y_1,1)\}$ is an independent set and $(w_1,1)(y_1,1)=1$. If $r=2$, then $(w_1,1)(w_2,1)=1$. Assume $r\geq3$, then doing the same along the path we get that 
    $(w_1,1)(w_2,1)=(y_1,1)(w_2,1)=(y_2,1)(w_2,1)=\cdots=(y_{r-2},1)(w_2,1)=((y_{r-2},1)(y_{r-1},1))((y_{r-1},1)(w_2,1))=1$. 
    \item If $e=\{(w,1),(w,2)\}$, we take $y_0y_1\cdots y_r$ the only path form $w$ to $x$ in $T_1$. For all $i$, we have that $\{(y_i,1),(y_i,2),(y_{i+1},2)\}$ is an independent set. Thus
    $$(y_i,1)(y_i,2)=((y_i,1)(y_i,2))((y_i,2)(y_{i+1},2))=(y_i,1)(y_{i+1},2)$$
    $$=((y_i,1)(y_{i+1},1))((y_{i+1},1)(y_{i+1},2))=(y_{i+1},1)(y_{i+1},2)$$
    Therefore $(w,1)(w,2)=(x,1)(x,2)=1$.
    \item If $e=\{(w_1,1),(w_2,2)\}$, by the previous cases we have that $$(w_1,1)(w_2,2)=((w_1,1)(w_2,1))((w_2,1)(w_2,2))=1.$$
\end{itemize}
the rest of possibilities are analogous. Therefore $H_\mathcal{T}$ is trivial.
\end{proof}
The lower bounds given in last Theorem are tight, to see this we take $C_5\times P_2\cong C_{10}$, by Theorem \ref{kozlovpaths} 
$I(C_5\times P_2)\simeq\mathbb{S}^2$.

\begin{theorem}\label{pathprodindusub}
$P_n\times P_m$ is in $\mathcal{SW}$ for all $n$ and $m\leq4$.
\end{theorem}
\begin{proof}
For $m=1$, we have that $P_n\times P_1\cong K_n^c$ and clearly $K_n^c$ is in $\mathcal{SW}$. For $m=2$, we have that $P_n\times P_2\cong2P_n$. Then, by Proposition \ref{propdsjotunion}, 
$P_n\times P_2$ is in $\mathcal{SW}$. 

For $m=3$, we know the result is true for $n\leq2$. Assume that $P_n\times P_3$ is in $\mathcal{SW}$. We take $H$ an induced subgraph of $P_{n+1}\times P_3$. If neither of the vertices $(1,1),(1,2),(1,3)$ 
are in $V(H)$, then $H$ is an induced subgraph of $P_{n+1}\times P_3-(1,1)-(1,2)-(1,3)\cong P_n\times P_3$ and $I(H)$ is contractible or it has the homotopy type of a wedge of spheres. 

If $(1,1)$ is a vertex of $H$, then $d_H((1,1))\leq1$ and $I(H)\simeq*$ or 
$I(H)\simeq\Sigma I(H-N_H[(2,2)])$. If $(1,2)$ is not in $V(H)-N_H[(2,2)]$, then $H-N_H[(2,2)]$ is an induced subgraph of a graph isomorphic to $P_n\times P_3$. 
If $(1,2)$ is in $V(H)-N_H[(2,2)]$, then $d_H((1,2))=d_{H-N_H[(2,2)]}((1,2))\leq2$ and $I(H-N_H[(2,2)])\simeq*$ or 
$I(H-N_H[(2,2)])\simeq\Sigma I(H-N_H[(2,2)]-N_H[(1,2)])$ or $$I(H-N_H[(2,2)])\simeq\Sigma I(H-N_H[(2,2)]-N_H[(1,2)]-(3,2)).$$ 
Regardless of the case, we got that $I(H)$ is either contractible or $I(H)$ has the homotopy 
type of the double suspension of a complex $I(H')$ where $H'$ is an induced subgraph of $P_{n+1}\times P_3-(1,1)-(1,2)-(1,3)\cong P_n\times P_3$. If $(1,3)$ is in $H$, we proceed as before. 
If $(1,2)$ is a vertex of $H$ and $(1,1),(1,3)$ are not in $V(H)$, we proceed as the last case when we assumed $(1,2)$ was a vertex of $H-N_H[(2,2)]$.

Lastly, assume $m=4$. We know that the theorem is true for $n\leq3$. Assume $P_n\times P_4$ is in $\mathcal{SW}$. We take $H$ an induced subgraph of $P_{n+1}\times P_4$. 
If neither of the vertices $(1,1),(1,2),(1,3),(1,4)$ are in $V(H)$, then $H$ is an induced subgraph of 
$P_{n+1}\times P_4-(1,1)-(1,2)-(1,3)-(1,4)\cong P_n\times P_4$ and there is nothing to prove. If the vertex
$(1,1)$ is in $H$, it has at most degree $1$, then $I(H)$ is contractible or $I(H)\simeq\Sigma I(H-N_H[(2,2)])$. If $(1,4)$ is in $H$, then 
$d_{H-N_H[(2,2)]}((1,4))=d_H((1,4))\leq1$ and $I(H-N_H[(2,2)])$ is  contractible or $$I(H-N_H[(2,2)])\simeq\Sigma I(H-N_H[(2,2)]-N_H[(2,3)])$$
with $H-N_H[(2,2)]-N_H[(2,3)]$ an induced subgraph of $P_{n+1}\times P_4-(1,1)-(1,2)-(1,3)-(1,4)\cong P_n\times P_4$.
Assume $(1,4)$ is not in $H$. If $(2,4)$ is a vertex of $H$, then $I(H-N_H[(2,2)])$ is contractible. Assume $(2,4)$ is not in $H$. If $(1,2)$ is not in $H$, then $H-N_H[(2,2)]$ is an induced subgraph 
of $P_{n+1}\times P_4-(1,1)-(1,2)-(1,3)-(1,4)\cong P_n\times P_4$. Assume $(1,2)$ is a vertex of $H$, then $d_{H-N_H[(2,2)]}((1,2))=d_H((1,2))\leq2$. If $d_H((1,2))=0$, then $I(H-N_H[(2,2)])$ is contractible. 
If $d_H((1,2))=1$, then $$I(H-N_H[(2,2)])\simeq\Sigma I(H-N_H[(2,2)]-N_H[(2,1)])\mbox{ or }$$  
$$I(H-N_H[(2,2)])\simeq\Sigma I(H-N_H[(2,2)]-N_H[(2,3)]).$$ 
Assume $d_H((1,2))=2$. If $(3,2)$ is a vertex of $H$, then $$I(H-N_H[(2,2)])\simeq\Sigma I(H-N_H[(2,2)]-N_H[(1,2)]-(3,2)).$$
Assume $(3,2)$ is not a vertex of $H$, then 
$$I(H-N_H[(2,2)])\simeq\Sigma I(H-N_H[(2,2)]-N_H[(1,2)]).$$
In either case $I(H)$ is contractible or has the homotopy type of the double suspension of a complex $I(H')$ where $H'$ is an induced subgraph of $P_{n+1}\times P_4-(1,1)-(1,2)-(1,3)-(1,4)\cong P_n\times P_4$.
If $(1,4)$ is in $H$, we proceed as before. If $(1,2)$ or $(1,3)$ are vertices of $H$ and $(1,1),(1,4)$ are not in $V(H)$, we proceed as the last case when we assumed $(1,2)$ was a vertex of $H-N_H[(2,2)]$.
\end{proof}

\begin{que}
Can Theorem \ref{pathprodindusub} be extended to any $m$?
\end{que}

Before we conclude this section,  we want to point out that while the homotopy type of the independence complex of the categorical product of paths was easy to determine, 
the homotopy type of $I(P_n\oblong P_m)$ is only  know for $1\leq m\leq6$ 
(see \citep{matsushitan4n5,matsushitan6}).

\section{Independence complexes of lexicographic products}
Given two graphs $G$ and $H$, the lexicographic product $G\circ H$ is given by taking a copy of $H$ for each vertex of $G$ and adding all the possible edges between two copies if the corresponding 
vertices are adjacent in $G$. We now will define an analogous general construction for simplicial complexes. Given a simplicial complex $K$ with vertex set $\underline{n}$ and 
$\underline{L}=\{L_1,\dots,L_n\}$ a family of simplicial complexes we define its \textit{polyhedral join} as the simplicial complex 
$$(\underline{L})^{*K}=\bigcup_{\sigma\in K}\left(\bigast_{i\in\sigma}L_i\right).$$
Then 
$$I(G\circ H)=\left(\underline{I(H)}\right)^{*I(G)}.$$
This approach  to study independence complexes was used in \citep{poljoin,MR4477852,okurajoinbosq}. In \citep{poljoin} it was proved that for any for any family $\underline{L}$ and a tree $T$:
\begin{itemize}
    \item If $T\ncong K_{1,n}$, then
    $$(\underline{L})^{*I(T)}\simeq\bigvee_{\sigma\in I(T)}\left(I\left(T-N_{T}[\sigma]\right)*\bigast_{i\in\sigma}L_i\right).$$
    \item If $T\cong K_{1,n}$ and $1$ is the vertex of degree $n$, then 
    $$(\underline{L})^{*I(T)}=L_1\sqcup\bigast_{i=2}^{n+1}L_i.$$
\end{itemize}
In \citep{poljoin} it was also showed that for a graph $G$ such that it is a cycle of length at least $5$, then 
$$(\underline{L})^{*I(G)}\simeq\bigvee_{\sigma\in I(G)}\left(I\left(G-N_{G}[\sigma]\right)*\bigast_{i\in\sigma}L_i\right)$$
for any any family $\underline{L}$. For a cycle of length $4$ it is clear that 
$$(\underline{L})^{*I(C_4)}=(L_1*L_3)\sqcup(L_2*L_4).$$
\begin{theorem}\label{teopoljoinsubind}
Let $H$ be a graph in $\mathcal{SW}$. Take $\mathcal{G}$ the graph family containing trees and cycles. 
For any $G\in\mathcal{G}$, $G\circ H$ is in $\mathcal{W}$.
\end{theorem}
\begin{proof}
Take $W$ a induced subgraph of $G\circ H$. Then, there is $G'$ an induced subgraph of $G$ and $\underline{H'}=\{H_{i_1},\dots,H_{i_j}\}$ a family of 
induced subgraphs of $H$, where $j$ is the order of $G'$, such that 
$$I(W)=\left(\underline{I(H')}\right)^{*I(G')}.$$
Now $G'$ is a cycle or a forest, thus $I(W)$ is contractible or it has the homotopy type of a wedge of spheres or it has the homotopy type of a disjoint union of wedges of spheres.
\end{proof}

The next proposition will allow us to show a result similar to Theorem \ref{teopoljoinsubind} form more general graphs by taking the suspensions of the complexes. 

\begin{prop}\citep{poljoin}\label{propsusppoljoin}
If $K$ is a simplicial complex with vertex set $\underline{n}$ and $(\underline{J})$ is a family of simplicial complexes, then 
$$\Sigma(\underline{J})^{*K}\simeq\bigvee_{\sigma\in K}\Sigma \mathrm{lk}(\sigma)\;\ast\bigast_{i\in\sigma}J_i.$$
\end{prop}

\begin{theorem}\label{teosubgrpsusp}
Let $G$ and $H$ be graphs of order at least two in $\mathcal{SW}$. Then for any induced subgraph $F$ in $G\circ H$, $\Sigma I(F)$ is contractible or has the homotopy type of a wedge of spheres.
\end{theorem}
\begin{proof}
We proceed as in the proof of Theorem \ref{teopoljoinsubind} and take the suspension 
$$\Sigma I(W)=\Sigma\left(\underline{H'}\right)^{*I(G')}.$$
By Proposition \ref{propsusppoljoin}, $\Sigma I(W)$ has the desired homotopy type.
\end{proof}

\begin{que}
Under the hypothesis of Theorem \ref{teosubgrpsusp}, Its is true that $G\circ H$ is in $\mathcal{W}$?
\end{que}

\section{Independence complexes of strong products}
In this section we focus on the independence complex of strong products of paths.
We start with two results that are easy to obtain from the definition of strong product.

\begin{lem}\label{lemclsngh}
Let $G$ and $H$ be graphs. If there are vertices $u,v$ in $V(G)$ such that $N_G[u]\subseteq N_G[v]$, then 
$N_{G\boxtimes H}[(u,x)]\subseteq N_{G\boxtimes H}[(v,x)]$ for any vertex $x$ in $V(H)$.
\end{lem}

\begin{cor}\label{corclsngh}
Let $G$ and $H$ be graphs. If there are vertices $u,v$ in $V(G)$ and $x,y$ in $V(H)$ such that $N_G[u]\subseteq N_G[v]$ and $N_H[x]\subseteq N_H[y]$, then 
$$N_{G\boxtimes H}[(u,x)]\subseteq N_{G\boxtimes H}[(u,y)]\subseteq N_{G\boxtimes H}[(v,y)],$$
$$N_{G\boxtimes H}[(u,x)]\subseteq N_{G\boxtimes H}[(v,x)]\subseteq N_{G\boxtimes H}[(v,y)].$$
\end{cor}

%\begin{cor}
%Let $G$ and $H$ be graphs. If there are vertices $u,v$ in $V(G)$ and $x,y$ in $V(H)$ such that $N_G[u]\subseteq N_G[v]$ and $N_H[x]\subseteq N_H[y]$. Taking $W=G\boxtimes H$, 
%we have that
%$$I(W)\simeq I(W-(u,y)-(v,x)-(v,y))\vee\Sigma I(W-N_{W}[(u,y)])\vee\Sigma I(W-N_{W}[(v,x)])\vee\Sigma I(W-N_{W}[(v,y)])$$
%\end{cor}

The graph $P_n\boxtimes P_m$  is known as the 
$n\times m$ King's graph-- given a chessboard of $n\times m$, an independent sets is given by a set of kings that cannot attack each other. 
If $n=1$ or $m=1$, then the graph is a path. For $n=1$ and $m=2$ the graph is $K_2$, with the independence complex equal 
to $\mathbb{S}^0$. For $n=m=2$ the graph is $K_4$, therefore its independence complex is the wedge of $3$ copies of $\mathbb{S}^0$.
\begin{prop}\label{proppnboxtimesp2}
For $n\geq4$
$$I(P_n\boxtimes P_2)\simeq\Sigma I(P_{n-2}\boxtimes P_2)\vee\bigvee_{2}\Sigma I(P_{n-3}\boxtimes P_2).$$
\end{prop}
\begin{proof}
Taking $G=P_n\boxtimes P_2$, we have that $N_G[(1,1)]=N_G[(1,2)]\subseteq N_G[(2,2)]$, by Lemma \ref{lemmvert} we have that 
$$I(G)\simeq I(G-(2,2))\vee\Sigma I(G-N_G[(2,2)]).$$ 
Now $G-N_G[(2,2)]\cong P_{n-3}\boxtimes P_2$. 
We take $G'=G-(2,2)$. 
In this new graph we have that $N_{G'}[(1,1)]=N_{G'}[(1,2)]\subseteq N_{G'}[(2,1)]$, so we can apply Lemma \ref{lemmvert} again, 
observe that $G'-N_{G'}[(2,1)]=G-N_{G'}[(2,2)]$ and that $G'-(2,1)\cong K_2\sqcup P_{n-2}\boxtimes P_2$.
\end{proof}

\begin{prop}\label{proppnboxtimesp3}
For $n\geq4$
$$I(P_n\boxtimes P_3)\simeq\Sigma I(P_{n-2}\boxtimes P_3)\vee\Sigma I(P_{n-3}\boxtimes P_3)\vee\Sigma^2I(P_{n-3}\boxtimes P_3).$$
\end{prop}
\begin{proof}
In $G=P_n\boxtimes P_3$, we have that $N_G[(1,3)]\subseteq N_G[(1,2)]$. Thus, by Lemma \ref{lemmvert}, 
$$I(G)\simeq (G-(1,2))\vee\Sigma I(G-N_G[(1,2)]),$$ 
where $G-N_G[(1,2)]\cong P_{n-2}\boxtimes P_3$. Now, in $G'=G-(1,2)$ we have that $N_G[(1,3)]\subseteq N_G[(2,2)]$, thus we can apply 
Lemma \ref{lemmvert} again. Now $G'-N_{G'}[(2,2)]\cong P_{n-3}\boxtimes P_3$, and in $G''=G'-(2,2)$ we have that 
$N_{G''}((1,3))\subseteq N_{G''}((3,3)),N_{G''}((3,2))$ and $N_{G''}((1,1))\subseteq N_{G''}((3,1)),N_{G''}((3,2))$, thus, applying 
Lemma \ref{lemmvert} 
$$I(G'')\simeq\Sigma^2I(P_{n-3}\boxtimes P_3).$$
\end{proof}

\begin{figure}
\centering
\begin{tikzpicture}[line cap=round,line join=round,>=triangle 45,x=2cm,y=2cm]
\clip(0.25,0.25) rectangle (2.75,2.25);
\draw (1,2)-- (1.5,1.5);
\draw (1.5,1.5)-- (2,1);
\draw (2,1)-- (2.5,0.5);
\draw (2.5,0.5)-- (2.5,1);
\draw (2.5,1)-- (2.5,1.5);
\draw (2.5,1.5)-- (2.5,2);
\draw (2.5,2)-- (2,2);
\draw (2,2)-- (1.5,2);
\draw (1.5,2)-- (1,2);
\draw (1,2)-- (1,1.5);
\draw (1,1.5)-- (1,1);
\draw (1,1)-- (1,0.5);
\draw (1,0.5)-- (1.5,0.5);
\draw (1.5,0.5)-- (2,0.5);
\draw (2,0.5)-- (2.5,0.5);
\draw (1,0.5)-- (1.5,1);
\draw (1.5,1)-- (2,1.5);
\draw (2,1.5)-- (2.5,2);
\draw (1,1.5)-- (1.5,2);
\draw (1,1)-- (1.5,1.5);
\draw (1.5,1.5)-- (2,2);
\draw (1.5,0.5)-- (2,1);
\draw (2,1)-- (2.5,1.5);
\draw (2,0.5)-- (2.5,1);
\draw (1.5,0.5)-- (1,1);
\draw (1,1.5)-- (1.5,1);
\draw (1.5,1)-- (2,0.5);
\draw (2.5,1)-- (2,1.5);
\draw (2,1.5)-- (1.5,2);
\draw (2,2)-- (2.5,1.5);
\draw (2.5,1.5)-- (2,1.5);
\draw (2,1.5)-- (2,2);
\draw (1.5,2)-- (1.5,1.5);
\draw (1.5,1.5)-- (2,1.5);
\draw (1.5,1.5)-- (1,1.5);
\draw (1.5,1.5)-- (1.5,1);
\draw (1.5,1)-- (1,1);
\draw (1.5,1)-- (2,1);
\draw (2,1)-- (2,1.5);
\draw (2,1)-- (2.5,1);
\draw (2,1)-- (2,0.5);
\draw (1.5,0.5)-- (1.5,1);
\draw (0.5,0.75)-- (1,1);
\draw (0.5,0.75)-- (1,0.5);
\begin{scriptsize}
\fill [color=black] (1,0.5) circle (1.5pt);
\fill [color=black] (1.5,0.5) circle (1.5pt);
\fill [color=black] (2,0.5) circle (1.5pt);
\fill [color=black] (2.5,0.5) circle (1.5pt);
\fill [color=black] (1,1) circle (1.5pt);
\fill [color=black] (1.5,1) circle (1.5pt);
\fill [color=black] (2,1) circle (1.5pt);
\fill [color=black] (2.5,1) circle (1.5pt);
\fill [color=black] (1,1.5) circle (1.5pt);
\fill [color=black] (1.5,1.5) circle (1.5pt);
\fill [color=black] (2,1.5) circle (1.5pt);
\fill [color=black] (2.5,1.5) circle (1.5pt);
\fill [color=black] (1,2) circle (1.5pt);
\fill [color=black] (1.5,2) circle (1.5pt);
\fill [color=black] (2,2) circle (1.5pt);
\fill [color=black] (2.5,2) circle (1.5pt);
\fill [color=black] (0.5,0.75) circle (1.5pt);
\end{scriptsize}
\end{tikzpicture}    
\caption{$Q_4$}
\label{q4}
\end{figure}

Now we define auxiliary graphs for the case $m=4$. For $0$, we take $Q_0=K_1$ and for $n\geq1$ we obtain $Q_n$ from $P_n\boxtimes P_4$ by 
adding a new vertex $x$ and making it adjacent to the vertices $(1,1)$ and $(1,2)$ (see Figure \ref{q4}).

\begin{prop}\label{proppnboxtimesp4}
For $n\geq4$
$$I(P_n\boxtimes P_4)\simeq\bigvee_3\Sigma^2I(P_{n-3}\boxtimes P_4)\vee\bigvee_2\Sigma^2I(Q_{n-3})\vee\bigvee_2\Sigma^3I(Q_{n-4}),$$
$$I(Q_n)\simeq\Sigma^2I(P_{n-2}\boxtimes P_4)\vee\bigvee_{3}\Sigma^2I(Q_{n-3})\vee\bigvee_{2}\Sigma^3I(Q_{n-4}).$$
\end{prop}
\begin{proof}
We start with $Q_n$. We have that $N_{Q_n}[(1,4)]\subseteq N_{Q_n}[(1,3)]$, thus we can apply Lemma \ref{lemmvert}. Now,
in $Q_n-N_{Q_n}[(1,3)]$ the only neighbor of $x$ is $(1,1)$, so by Lemma \ref{lemmvert} 
$$I(Q_n-N_{Q_n}[(1,3)])\simeq I(Q_n-N_{Q_n}[(1,3)]-(1,2))\cong\Sigma I(P_{n-2}\boxtimes P_4).$$
Now, in $G=Q_n-(1,3)$ we have that $N_{G}[(1,4)]\subseteq N_{G}[(2,3)]$ and we can apply Lemma \ref{lemmvert}. 
In $G'=G-N_{G}[(2,3)]$ we have that $N_{G'}(x)\subseteq N_{G'}(2,1)$, thus $I(G')\simeq I(G'-(2,1))\cong\Sigma I(Q_{n-3})$. 
Now, in $G''=Q_n-(1,3)-(2,3)$ by Lemma \ref{lemmvert} we have that 
$$I(G'')\simeq I(G''-(3,4)-(3,3)-(2,2)-(2,1))\simeq\bigvee_2\Sigma^2I(H)$$
where $H=G''-(3,4)-(3,3)-(2,2)-(2,1)-(1,2)-(1,1)-(1,4)-(2,4)-x$. In $H$ we have that 
$N_H[(3,1)]\subseteq N_H[(3,2)]$, thus 
$$I(H)\simeq I(H-(3,2))\vee\Sigma I(H-N_H[(3,2)])\cong I(Q_{n-3})\vee\Sigma I(Q_{n-4}).$$

For  $F=P_n\boxtimes P_4$, $N_F[(1,4)]\subseteq N_F[(1,3)]$, therefore 
$$I(F)\simeq I(F-(1,3))\vee\Sigma I(F-N_F[(1,3)]).$$
In $F-N_F[(1,3)]$ the only neighbor of $(1,1)$ is $(2,1)$, thus 
$$I(F-N_F[(1,3)])\simeq I(F-N_F[(1,3)]-(3,2)-(3,1))$$
where $F-N_F[(1,3)]-(3,2)-(3,1)\cong H+ K_2$ and $H$ the graph we obtain while working with $Q_n$, thus 
$$I(F-N_F[(1,3)])\simeq\Sigma I(Q_{n-3})\vee\Sigma^2I(Q_{n-4}).$$
We take $F'=F-(1,3)$. We have that $N_{F'}[(1,4)]\subseteq N_{F'}[(2,3)]$, thus 
$$I(F')\simeq I(F'-(2,3))\vee\Sigma I(F'-N_{F'}[(2,3)]).$$
In $F'-N_{F'}[(2,3)]$, the only neighbor of $(1,1)$ is $(2,1)$, thus 
$$I(F'-N_{F'}[(2,3)])\simeq I(F'-N_{F'}[(2,3)]-(3,1))\cong\Sigma I(P_{n-3}\boxtimes P_4).$$
Now, in $F'-(2,3)$ the only neighbor of $(1,4)$ is $(2,4)$, so we can erase the vertices $(3,4),(3,3)$ without changing the homotopy type-- 
observe that the graph obtained has $2$ connected components, one of these components is isomorphic to $K_2$. We call this new graph $F''$. In $F''$, 
$N_{F''}[(1,2)]\subseteq N_{F''}[(2,2)]$, thus 
$$I(F'')\simeq I(F''-(2,2))\vee\Sigma I(F''-N_{F''}[(2,2)]).$$
Observe that $F''-N_{F''}[(2,2)]\cong K_2+ P_{n-3}\boxtimes P_4$. In $F'''=F''-(2,2)$, we have that 
$N_{F'''}[(1,2)]\subseteq N_{F'''}[(2,1)]$, thus 
$$I(F''')\simeq I(F'''-(2,1))\vee \Sigma I(F'''-N_{F'''}[(2,1)])$$
where $F'''-N_{F'''}[(2,1)]\cong K_2+ P_{n-3}\boxtimes P_4$ and $F'''-(2,1)\cong2K_2+ H$ with $H$ the same graph as before.
\end{proof}

Using Lemma \ref{lemmvert} it is easy to obtain the following proposition. 
\begin{prop}\label{propp123boxtimesp234}
$$I(P_n\boxtimes P_2)\simeq\left\lbrace\begin{array}{cc}
\mathbb{S}^0  & \mbox{if }n=1 \\
\displaystyle\bigvee_3\mathbb{S}^0&  \mbox{if }n=2\\
\displaystyle\mathbb{S}^1\vee\bigvee_{2}\mathbb{S}^0& \mbox{if }n=3
\end{array}\right.,\;\;I(P_n\boxtimes P_3)\simeq\left\lbrace\begin{array}{cc}
\mathbb{S}^0  & \mbox{if }n=1 \\
\displaystyle\mathbb{S}^1\vee\bigvee_{2}\mathbb{S}^0&  \mbox{if }n=2\\
\displaystyle\bigvee_{2}\mathbb{S}^1\vee\mathbb{S}^0&  \mbox{if }n=3
\end{array}\right.$$
$$I(P_n\boxtimes P_4)\simeq\left\lbrace\begin{array}{cc}
\ast  & \mbox{if }n=1 \\
\displaystyle\bigvee_5\mathbb{S}^1&  \mbox{if }n=2\\
\displaystyle\bigvee_{2}\mathbb{S}^2\vee\bigvee_{3}\mathbb{S}^1& \mbox{if }n=3
\end{array}\right.,\;\;I(Q_n)\simeq\left\lbrace\begin{array}{cc}
\ast  & \mbox{if }n=0 \\
\displaystyle\mathbb{S}^1&  \mbox{if }n=1\\
\displaystyle\bigvee_{4}\mathbb{S}^1&  \mbox{if }n=2\\
\displaystyle\bigvee_{2}\mathbb{S}^2&  \mbox{if }n=3
\end{array}\right.$$
\end{prop}

\begin{cor}
For $2\leq m\leq4$, $I(P_n\boxtimes P_m)$ has the homotopy type of a wedge of spheres.
\end{cor}
\begin{proof}
The result follows from Propositions \ref{proppnboxtimesp2},\ref{proppnboxtimesp3},\ref{proppnboxtimesp4} and \ref{propp123boxtimesp234}.
\end{proof}

Now, for any $n$, we define the following polynomials:
\begin{itemize}
    \item For $m=2$, $\displaystyle f_n(x)=\sum_{i=0}\Tilde{\beta}_i(I(P_n\boxtimes P_2))x^i$
    \item For $m=3$, $\displaystyle g_n(x)=\sum_{i=0}\Tilde{\beta}_i(I(P_n\boxtimes P_3))x^i$
\end{itemize}
where $\Tilde{\beta}_i$ is the reduced $i$-Betti number.
Then $f_1(x)=1,\;f_2(x)=3,\;f_3(x)=x+2,\;g_1(x)=1,\;g_2(x)=x+2,\;g_3(x)=3x+1$ and for all $n\geq4$
$$f_n(x)=xf_{n-2}(x)+2xf_{n-3}(x),$$
$$g_n(x)=xf_{n-2}(x)+(x^2+x)f_{n-3}(x).$$
From this it is easy to get the following proposition.
\begin{prop}
The generating function for the $f_i's$ is 
$$F(t)=\frac{2t(t+1)(t+\frac{1}{2})}{1-xt^2-2xt^3}$$
and for the $g_i's$ the generating function is 
$$G(t)=\frac{(3x+1)t^3+2t^2+t}{1-xt^2-(x+x^2)t^3}.$$
\end{prop}

Now we give the analogous theorem to Theorem \ref{pathprodindusub} for the strong product.
\begin{theorem}\label{theosubgrapdstrongprod}
$P_n\boxtimes P_m$ is in $\mathcal{SW}$ for any $n$ and $1\leq m\leq4$.
\end{theorem}
\begin{proof}
For $m=1$, any induced subgraph is a disjoint union of paths, therefore its independence complex is the join of independence complexes 
of paths, by Theorem \ref{kozlovpaths} each one is contractible or has the homotopy type of a single sphere. 

Assume $m=2$ and that the result is false in this case. 
Let $n$ be the first positive integer for which there is an induced subgraph $H$ of $P_n\boxtimes P_2$
such that $I(H)$ is not contractible nor has the homotopy type of a wedge of spheres. 
Let $H$ be a vertex minimal counterexample. Now, by minimality of $n$ 
either $(1,1)$ or $(1,2)$ is in $H$ and not both, as by Lemma \ref{lemmvert} $H$ would not be a counterexample. 
Because $I(H)$ is not contractible, $(2,1)$ or $(2,2)$ must be in $H$, in any case Lemma \ref{lemmvert} tell us that 
$H$ cannot be a counterexample.

Next we take $m=3$. As before, we suppose the result is false and we take $n$ the first positive integer for which the results fails. 
By the work done we know $n\geq3$. We take $H$ a vertex minimal counterexample for $n$, this means $H$ is an induced subgraph of 
$P_n\boxtimes P_3$ such that $I(H)$ is not contractible nor has the homotopy type of a wedge of spheres. By minimality of $n$, $H$ 
must have a vertex of the form $(1,j)$. If $j=1$, because $I(H)$ is not contractible at least one of the vertices $(1,2),(2,1),(2,2)$ must be 
in $H$, then $N_H[(1,1)]$ is contained in the close neighborhood of any of these vertices. Using Lemma \ref{lemmvert} and the minimality 
of $H$ we see that $H$ could not be a counterexample. Thus $j\neq1$. With a similar argument we can see that $j\neq3$. 
Then $(1,2)$ is a vertex of $H$.
Now $(2,2)$ is not a vertex of $H$, otherwise $N_H[(1,2)]\subseteq N_H[(2,2)]$. 
Now $(3,2)$ is not a vertex of $H$, because if $(3,2)$ is a vertex of $H$, then 
$N_H(1,2)\subseteq N_H(3,2)$ and $I(H)\simeq I(H-(3,2))$. With the same type of arguments we can see that $\delta(H)\geq2$ and therefore $(2,1),(2,3),(3,1),(3,3)$ are 
vertices of $H$. Now $n\geq4$, otherwise $H\cong P_5$ and $H$ would not be a counterexample. 

Now, there are $5$ possibilities: 
\begin{enumerate}
    \item $(4,2)$ is a vertex of $H$ and $(4,1),(4,3)$ are not. In this case $H-N_H[(4,2)]$ the disjoin union of two graphs, one of which is
    isomorphic to $P_3$ and $H-(4,2)$ the disjoin union of two graphs, one of which is isomorphic to $P_5$, In this graphs 
    the $P_3$ is contain in the $P_5$ and the other graphs are the same, thus the inclusion 
    $$I(H-N_H[(4,2)])\longhookrightarrow I(H-(4,2))$$
    is null-homotopic, because $I(P_3)\simeq\mathbb{S}^0$ and $I(P_5)\simeq\mathbb{S}^1$. Therefore 
    $$I(H)\simeq\Sigma I(H-N_H[(4,2)])\vee I(H-(4,2))$$ 
    and $H$ is not a counterexample.
    \item $(4,1),(4,2)$ are vertices of $H$ and $(4,3)$ is not. Is $H-N_H[(4,1)]$ the disjoin union of two graphs, one of which is
    isomorphic to $P_4$, thus $I(H-N_H[(4,1)])\simeq*$ and $I(H)\simeq I(H-(4,1))$. Then $H$ can not be a counterexample.
    \item $(4,2),(4,3)$ are vertices of $H$ and $(4,1)$ is not. This case is equivalent to the last case.
    \item $(4,1),(4,2),(4,3)$ are vertices of $H$. By Theorem \ref{edgesubd} $I(H)\simeq\Sigma I(T)$, where 
    $$T=(H-(1,2)-(2,1)-(2,3))+(3,1)(3,3).$$
    Now, by Lemma \ref{lemmvert}, $I(T)\simeq I(T-(4,2))\vee\Sigma I(T-N_T[(4,2)])$. Observe that $T-N_T[(4,2)]$ is an induced subgraph of 
    $H$, therefore, by the  minimality of $H$, $I(T-N_T[(4,2)])$ is contractible or has the homotopy type of a wedge of spheres. 
    We take $W=T-(4,2)$.
    \item $(4,1),(4,3)$ are vertices of $H$ and $(4,2)$ is not. By Theorem \ref{edgesubd} $I(H)\simeq\Sigma I(W)$, where 
    $$W=(H-(1,2)-(2,1)-(2,3))+(3,1)(3,3).$$
\end{enumerate}
For either of the last two cases, if $n=4$, then $I(W)\simeq*$ and $H$ is not a counterexample. Assume $n\geq5$, by Theorem \ref{theostrclstrdgr2}
%then 
%we have that 
%$$I(W)\simeq\Sigma\left(I(W-N_W[(3,3)]-N_W[(4,3)])\cup I(W-N_W[(3,3)]-N_W[(3,1)])\right).$$
%Now, it is not hard to see that
%$$I(W-N_W[(3,3)]-N_W[(4,3)])\cup I(W-N_W[(3,3)]-N_W[(3,1)])=I(Q)$$
we have that $I(W)\simeq\Sigma I(Q)$
where $Q=(W-(3,3)-(3,1)-(4,3))+(4,1)(5,3)$ if $(5,3)$ is a vertex of $W$ and if not is not a vertex of $W$
then $Q=W-(3,3)-(3,1)-(4,3)$. In either of case, $Q$ is isomorphic to a proper induced subgraph of $H$. Therefore there is no counterexample.

Lastly we take $m=4$. As before, we suppose the result is false and we take $n$ the first positive integer for which the results fails. 
By the work done we know that $n\geq4$. Let $H$ be a vertex minimal counterexample. If $(1,1)$ is in $H$, then it must have degree at least $1$, if not $I(H)$ would be contractible. 
Let $x$ be a neighbor of $(1,1)$, then $N_H[(1,1)]\subseteq N_H[x]$ and, by Lemma \ref{lemmvert}, 
$$I(H)\simeq I(H-x)\vee\Sigma I(H-N_H[x])$$
which would imply that $H$ is not a counterexample. Therefore $(1,1)$ is not a vertex of $H$. Similarly $(1,4)$ is not a vertex of $H$. 
Now, at least one of $(1,2)$ and $(1,3)$ is in $H$, otherwise $H$ would be an induced subgraph of 
$P_{n+1}\boxtimes P_4-(1,1)-(1,2)-(1,3)-(1,4)\cong P_n\boxtimes P_4$. First assume both vertices are in $V(H)$, then 
$(2,2)$ and $(2,3)$ are not in $V(H)$. Both $(1,2)$ and $(1,3)$ must have degree $2$, otherwise $I(H)\cong \Sigma I(H-N_H[(1,2)])$ or $I(H)\cong \Sigma I(H-N_H[(1,3)])$ and $H$ would not be a 
counterexample. 
Then $H$ is like Figure \ref{graphh}, where we know that the solid vertices and edges are in $H$ and the other may be-- this will be the convention for all the figures. By Theorem \ref{theostrclstrdgr2} $I(H)\simeq\Sigma I(W)$ where $W$ is the graph in 
Figure \ref{graphw} obtained from $H-(1,2)-(1,3)-(2,4)$ by adding the edges $\{(2,1),(3,b)\}$ if $(3,b)$ is in $H$. If $(3,1)$ or $(3,4)$ are not in $W$, the $W$ is an induced subgraph of a graph isomorphic 
to $P_n\boxtimes P_4$. Then $(3,1)$ and $(3,4)$ are in $W$. 

If $(3,2)$ is in $V(W)$, then $I(W)\simeq I(W-(3,2))\vee\Sigma I(W-N_W[(3,2)])$. Now, $W-N_W[(3,2)]$ is an induced subgraph of a graph isomorphic 
to $P_n\boxtimes P_4$, thus we only are interested in $W-(3,2)$. If $(3,3)$ is in $W$, then $I(W-(3,3))\simeq I(W-(3,2)-(3,3))\vee\Sigma I(W-(3,2)-N_W[(3,3)])$ and again, we are only interested in 
$W-(3,2)-(3,3)$. So we can assume that neither $(3,2)$ nor $(3,3)$ are in $W$. If both $(4,1)$ and $(4,2)$ are in $W$, then $I(W)\simeq I(W-(4,2))\vee\Sigma I(W-N_W[(4,2)])$, where as before we are only 
interested in $W-(4,2)$. Thus we can assume that only one of $(4,1)$ and $(4,2)$ is in $W$. By the same argument we can assume that only one of $(4,3)$ and $(4,4)$ is in $W$. Up to isomorphism there are 
only three possible cases shown in Figures \ref{casei}, \ref{caseii} and \ref{caseiii}.
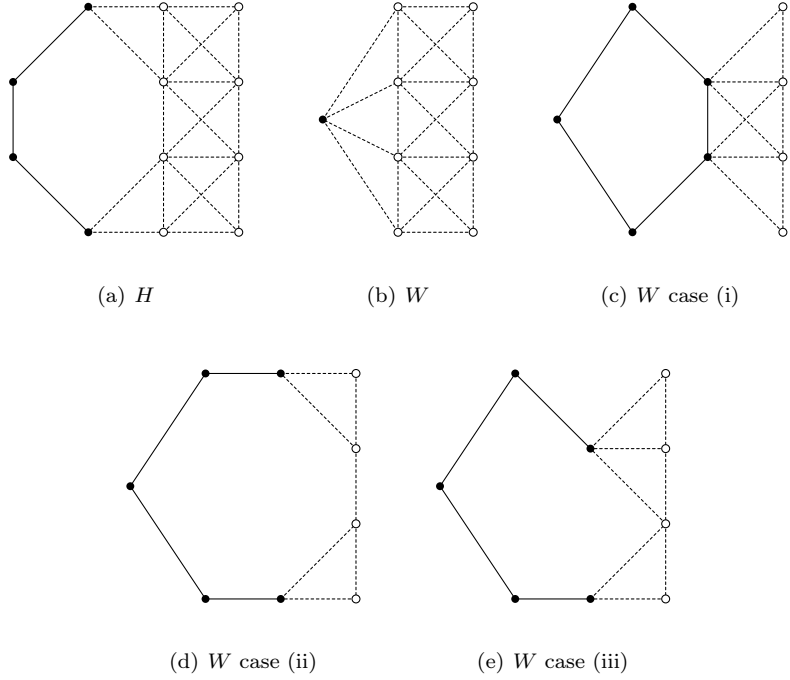
\begin{figure}
\centering
\subfigure[$H$]{\begin{tikzpicture}[line cap=round,line join=round,>=triangle 45,x=1.0cm,y=1.0cm]
\clip(0.5,0.5) rectangle (4.5,4.5);
\draw (2,4)-- (1,3);
\draw (1,3)-- (1,2);
\draw (1,2)-- (2,1);
\draw [dash pattern=on 1pt off 1pt] (2,4)-- (3,4);
\draw [dash pattern=on 1pt off 1pt] (3,3)-- (2,4);
\draw [dash pattern=on 1pt off 1pt] (3,4)-- (3,3);
\draw [dash pattern=on 1pt off 1pt] (2,1)-- (3,1);
\draw [dash pattern=on 1pt off 1pt] (3,1)-- (3,2);
\draw [dash pattern=on 1pt off 1pt] (3,2)-- (2,1);
\draw [dash pattern=on 1pt off 1pt] (3,2)-- (3,3);
\draw [dash pattern=on 1pt off 1pt] (3,4)-- (4,4);
\draw [dash pattern=on 1pt off 1pt] (4,4)-- (3,3);
\draw [dash pattern=on 1pt off 1pt] (3,3)-- (4,3);
\draw [dash pattern=on 1pt off 1pt] (4,3)-- (3,4);
\draw [dash pattern=on 1pt off 1pt] (4,4)-- (4,3);
\draw [dash pattern=on 1pt off 1pt] (4,3)-- (4,2);
\draw [dash pattern=on 1pt off 1pt] (4,2)-- (4,1);
\draw [dash pattern=on 1pt off 1pt] (4,1)-- (3,1);
\draw [dash pattern=on 1pt off 1pt] (3,1)-- (4,2);
\draw [dash pattern=on 1pt off 1pt] (4,1)-- (3,2);
\draw [dash pattern=on 1pt off 1pt] (3,2)-- (4,2);
\draw [dash pattern=on 1pt off 1pt] (4,2)-- (3,3);
\draw [dash pattern=on 1pt off 1pt] (4,3)-- (3,2);
\begin{scriptsize}
\fill [color=black] (1,2) circle (1.5pt);
\fill [color=black] (1,3) circle (1.5pt);
\fill [color=black] (2,4) circle (1.5pt);
\fill [color=black] (2,1) circle (1.5pt);
\fill [color=white] (3,3) circle (1.5pt);
\draw [color=black] (3,3) circle (1.5pt);
\fill [color=white] (3,4) circle (1.5pt);
\draw [color=black] (3,4) circle (1.5pt);
\fill [color=white] (3,2) circle (1.5pt);
\draw [color=black] (3,2) circle (1.5pt);
\fill [color=white] (3,1) circle (1.5pt);
\draw [color=black] (3,1) circle (1.5pt);
\fill [color=white] (4,4) circle (1.5pt);
\draw [color=black] (4,4) circle (1.5pt);
\fill [color=white] (4,3) circle (1.5pt);
\draw [color=black] (4,3) circle (1.5pt);
\fill [color=white] (4,2) circle (1.5pt);
\draw [color=black] (4,2) circle (1.5pt);
\fill [color=white] (4,1) circle (1.5pt);
\draw [color=black] (4,1) circle (1.5pt);
\end{scriptsize}
\end{tikzpicture}\label{graphh}}
\subfigure[$W$]{\begin{tikzpicture}[line cap=round,line join=round,>=triangle 45,x=1.0cm,y=1.0cm]
\clip(1.5,0.5) rectangle (4.5,4.5);
\draw [dash pattern=on 1pt off 1pt] (3,4)-- (3,3);
\draw [dash pattern=on 1pt off 1pt] (2,2.5)-- (3,1);
\draw [dash pattern=on 1pt off 1pt] (3,1)-- (3,2);
\draw [dash pattern=on 1pt off 1pt] (3,2)-- (2,2.5);
\draw [dash pattern=on 1pt off 1pt] (3,2)-- (3,3);
\draw [dash pattern=on 1pt off 1pt] (3,4)-- (4,4);
\draw [dash pattern=on 1pt off 1pt] (4,4)-- (3,3);
\draw [dash pattern=on 1pt off 1pt] (3,3)-- (4,3);
\draw [dash pattern=on 1pt off 1pt] (4,3)-- (3,4);
\draw [dash pattern=on 1pt off 1pt] (4,4)-- (4,3);
\draw [dash pattern=on 1pt off 1pt] (4,3)-- (4,2);
\draw [dash pattern=on 1pt off 1pt] (4,2)-- (4,1);
\draw [dash pattern=on 1pt off 1pt] (4,1)-- (3,1);
\draw [dash pattern=on 1pt off 1pt] (3,1)-- (4,2);
\draw [dash pattern=on 1pt off 1pt] (4,1)-- (3,2);
\draw [dash pattern=on 1pt off 1pt] (3,2)-- (4,2);
\draw [dash pattern=on 1pt off 1pt] (4,2)-- (3,3);
\draw [dash pattern=on 1pt off 1pt] (4,3)-- (3,2);
\draw [dash pattern=on 1pt off 1pt] (2,2.5)-- (3,3);
\draw [dash pattern=on 1pt off 1pt] (2,2.5)-- (3,4);
\begin{scriptsize}
\fill [color=black] (2,2.5) circle (1.5pt);
\fill [color=white] (3,3) circle (1.5pt);
\draw [color=black] (3,3) circle (1.5pt);
\fill [color=white] (3,4) circle (1.5pt);
\draw [color=black] (3,4) circle (1.5pt);
\fill [color=white] (3,2) circle (1.5pt);
\draw [color=black] (3,2) circle (1.5pt);
\fill [color=white] (3,1) circle (1.5pt);
\draw [color=black] (3,1) circle (1.5pt);
\fill [color=white] (4,4) circle (1.5pt);
\draw [color=black] (4,4) circle (1.5pt);
\fill [color=white] (4,3) circle (1.5pt);
\draw [color=black] (4,3) circle (1.5pt);
\fill [color=white] (4,2) circle (1.5pt);
\draw [color=black] (4,2) circle (1.5pt);
\fill [color=white] (4,1) circle (1.5pt);
\draw [color=black] (4,1) circle (1.5pt);
\end{scriptsize}
\end{tikzpicture}\label{graphw}}
\subfigure[$W$ case (i)]{\begin{tikzpicture}[line cap=round,line join=round,>=triangle 45,x=1.0cm,y=1.0cm]
\clip(0.5,0.5) rectangle (4.5,4.5);
\draw (2,4)-- (3,3);
\draw (2,4)-- (1,2.5);
\draw (1,2.5)-- (2,1);
\draw (2,1)-- (3,2);
\draw [dash pattern=on 1pt off 1pt] (3,3)-- (4,4);
\draw [dash pattern=on 1pt off 1pt] (4,4)-- (4,3);
\draw [dash pattern=on 1pt off 1pt] (4,3)-- (3,3);
\draw [dash pattern=on 1pt off 1pt] (4,3)-- (4,2);
\draw [dash pattern=on 1pt off 1pt] (4,2)-- (4,1);
\draw [dash pattern=on 1pt off 1pt] (4,1)-- (3,2);
\draw [dash pattern=on 1pt off 1pt] (3,2)-- (4,2);
\draw [dash pattern=on 1pt off 1pt] (4,3)-- (3,2);
\draw [dash pattern=on 1pt off 1pt] (4,2)-- (3,3);
\draw (3,3)-- (3,2);
\begin{scriptsize}
\fill [color=black] (1,2.5) circle (1.5pt);
\fill [color=black] (2,4) circle (1.5pt);
\fill [color=black] (2,1) circle (1.5pt);
\fill [color=black] (3,3) circle (1.5pt);
\fill [color=black] (3,2) circle (1.5pt);
\fill [color=white] (4,4) circle (1.5pt);
\draw [color=black] (4,4) circle (1.5pt);
\fill [color=white] (4,3) circle (1.5pt);
\draw [color=black] (4,3) circle (1.5pt);
\fill [color=white] (4,2) circle (1.5pt);
\draw [color=black] (4,2) circle (1.5pt);
\fill [color=white] (4,1) circle (1.5pt);
\draw [color=black] (4,1) circle (1.5pt);
\end{scriptsize}
\end{tikzpicture}\label{casei}}
\subfigure[$W$ case (ii)]{\begin{tikzpicture}[line cap=round,line join=round,>=triangle 45,x=1.0cm,y=1.0cm]
\clip(0.5,0.5) rectangle (4.5,4.5);
\draw (2,4)-- (3,4);
\draw (2,4)-- (1,2.5);
\draw (1,2.5)-- (2,1);
\draw (2,1)-- (3,1);
\draw [dash pattern=on 1pt off 1pt] (3,4)-- (4,4);
\draw [dash pattern=on 1pt off 1pt] (4,4)-- (4,3);
\draw [dash pattern=on 1pt off 1pt] (4,3)-- (3,4);
\draw [dash pattern=on 1pt off 1pt] (4,3)-- (4,2);
\draw [dash pattern=on 1pt off 1pt] (4,2)-- (4,1);
\draw [dash pattern=on 1pt off 1pt] (4,1)-- (3,1);
\draw [dash pattern=on 1pt off 1pt] (3,1)-- (4,2);
\begin{scriptsize}
\fill [color=black] (1,2.5) circle (1.5pt);
\fill [color=black] (2,4) circle (1.5pt);
\fill [color=black] (2,1) circle (1.5pt);
\fill [color=black] (3,4) circle (1.5pt);
\fill [color=black] (3,1) circle (1.5pt);
\fill [color=white] (4,4) circle (1.5pt);
\draw [color=black] (4,4) circle (1.5pt);
\fill [color=white] (4,3) circle (1.5pt);
\draw [color=black] (4,3) circle (1.5pt);
\fill [color=white] (4,2) circle (1.5pt);
\draw [color=black] (4,2) circle (1.5pt);
\fill [color=white] (4,1) circle (1.5pt);
\draw [color=black] (4,1) circle (1.5pt);
\end{scriptsize}
\end{tikzpicture}\label{caseii}}
\subfigure[$W$ case (iii)]{\begin{tikzpicture}[line cap=round,line join=round,>=triangle 45,x=1.0cm,y=1.0cm]
\clip(0.5,0.5) rectangle (4.5,4.5);
\draw (2,4)-- (3,3);
\draw (2,4)-- (1,2.5);
\draw (1,2.5)-- (2,1);
\draw (2,1)-- (3,1);
\draw [dash pattern=on 1pt off 1pt] (3,3)-- (4,4);
\draw [dash pattern=on 1pt off 1pt] (4,4)-- (4,3);
\draw [dash pattern=on 1pt off 1pt] (4,3)-- (3,3);
\draw [dash pattern=on 1pt off 1pt] (4,3)-- (4,2);
\draw [dash pattern=on 1pt off 1pt] (4,2)-- (4,1);
\draw [dash pattern=on 1pt off 1pt] (4,1)-- (3,1);
\draw [dash pattern=on 1pt off 1pt] (3,1)-- (4,2);
\draw [dash pattern=on 1pt off 1pt] (4,2)-- (3,3);
\begin{scriptsize}
\fill [color=black] (1,2.5) circle (1.5pt);
\fill [color=black] (2,4) circle (1.5pt);
\fill [color=black] (2,1) circle (1.5pt);
\fill [color=black] (3,3) circle (1.5pt);
\fill [color=black] (3,1) circle (1.5pt);
\fill [color=white] (4,4) circle (1.5pt);
\draw [color=black] (4,4) circle (1.5pt);
\fill [color=white] (4,3) circle (1.5pt);
\draw [color=black] (4,3) circle (1.5pt);
\fill [color=white] (4,2) circle (1.5pt);
\draw [color=black] (4,2) circle (1.5pt);
\fill [color=white] (4,1) circle (1.5pt);
\draw [color=black] (4,1) circle (1.5pt);
\end{scriptsize}
\end{tikzpicture}\label{caseiii}}
\caption{Auxiliary graphs used in Theorem \ref{theosubgrapdstrongprod} (I).}
\end{figure}

\textbf{Case (i):} By Theorem \ref{theostrclstrdgr2}, we have that $I(W)\simeq\Sigma I(W-N_W[(2,1)])$, where $W-N_W[(2,1)]$ 
is an induced subgraph of a graph isomorphic to $P_n\boxtimes P_4$. Therefore this case is not possible.

\textbf{Case (ii):} By Theorem \ref{edgesubd}, $I(W)\simeq\Sigma I(R)$ where $R$ is obtained from $W-N_W[(2,1)]$ by adding the edge $\{(4,1),(4,4)\}$ (see Figure \ref{graphr}). 
With arguments similar at what 
we have done, we can assume that only one vertex of $(5,1)$ and $(5,2)$ is in $R$, the same for $(5,3)$ and $(5,4)$. Therefore there are four possible cases shown in Figures \ref{casea}, \ref{caseb} and 
\ref{casec}. \textbf{\textit{Case (a):}} This case is not possible because $R$ would be an induced subgraph of a graph isomorphic to $P_{n-3}\boxtimes P_4$. \textbf{\textit{Case (b):}} In this case we have that 
$N_R((4,1))\subseteq N_R((5,3))$ and $N_R((4,4))\subseteq N_R((5,2))$, thus $I(R)\simeq I(R-(5,2)-(5,3))$. This can not happen as $R-(5,2)-(5,3)$ is the disjoint union of an edge and an induced subgraph 
of a graph isomorphic to $P_{n-5}\boxtimes P_4$. \textbf{\textit{Case (c):}} With the same arguments as before we can assume that only one vertex of $(6,1)$ and $(6,2)$ is in $R$, the same for $(6,3)$ and $(6,4)$.
Thus there are four cases shown in Figures \ref{casec1}, \ref{casec2}, \ref{casec3} and \ref{casec4}. The first two cases are not possible because $R$ would be an induced subgraph of a graph isomorphic to 
$P_{n-3}\boxtimes P_4$. In other two cases, by Theorem \ref{theostrclstrdgr2}, $I(R)\simeq\Sigma I(R-(4,1)-(4,4)-(5,1))$, but $R-(4,1)-(4,4)-(5,1)$ is an induced subgraph of a graph isomorphic to 
$P_{n-4}\boxtimes P_4$. From all that we get that case (ii) is not possible.
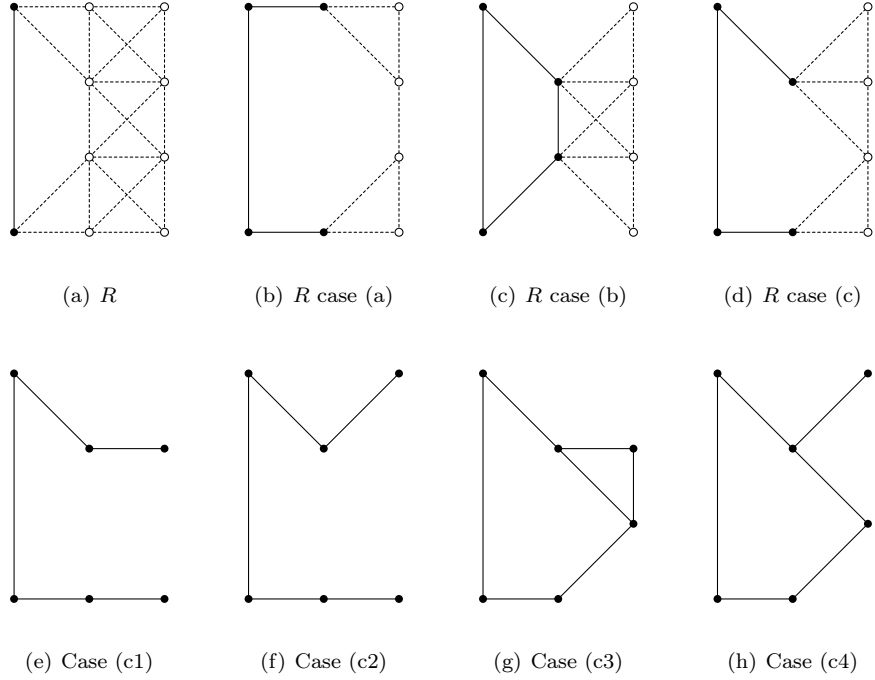
\begin{figure}
\centering
\subfigure[$R$]{\begin{tikzpicture}[line cap=round,line join=round,>=triangle 45,x=1.0cm,y=1.0cm]
\clip(0.5,0.5) rectangle (3.5,4.5);
\draw [dash pattern=on 1pt off 1pt] (1,4)-- (2,4);
\draw [dash pattern=on 1pt off 1pt] (2,4)-- (2,3);
\draw [dash pattern=on 1pt off 1pt] (2,3)-- (1,4);
\draw [dash pattern=on 1pt off 1pt] (2,3)-- (2,2);
\draw [dash pattern=on 1pt off 1pt] (2,2)-- (2,1);
\draw [dash pattern=on 1pt off 1pt] (2,1)-- (1,1);
\draw [dash pattern=on 1pt off 1pt] (1,1)-- (2,2);
\draw (1,4)-- (1,1);
\draw [dash pattern=on 1pt off 1pt] (2,4)-- (3,4);
\draw [dash pattern=on 1pt off 1pt] (3,4)-- (2,3);
\draw [dash pattern=on 1pt off 1pt] (2,3)-- (3,3);
\draw [dash pattern=on 1pt off 1pt] (3,3)-- (2,4);
\draw [dash pattern=on 1pt off 1pt] (3,4)-- (3,3);
\draw [dash pattern=on 1pt off 1pt] (3,3)-- (3,2);
\draw [dash pattern=on 1pt off 1pt] (3,2)-- (2,2);
\draw [dash pattern=on 1pt off 1pt] (2,2)-- (3,3);
\draw [dash pattern=on 1pt off 1pt] (3,2)-- (2,3);
\draw [dash pattern=on 1pt off 1pt] (3,2)-- (2,1);
\draw [dash pattern=on 1pt off 1pt] (3,1)-- (2,2);
\draw [dash pattern=on 1pt off 1pt] (3,1)-- (3,2);
\draw [dash pattern=on 1pt off 1pt] (3,1)-- (2,1);
\begin{scriptsize}
\fill [color=black] (1,4) circle (1.5pt);
\fill [color=black] (1,1) circle (1.5pt);
\fill [color=white] (2,4) circle (1.5pt);
\draw [color=black] (2,4) circle (1.5pt);
\fill [color=white] (2,3) circle (1.5pt);
\draw [color=black] (2,3) circle (1.5pt);
\fill [color=white] (2,2) circle (1.5pt);
\draw [color=black] (2,2) circle (1.5pt);
\fill [color=white] (2,1) circle (1.5pt);
\draw [color=black] (2,1) circle (1.5pt);
\fill [color=white] (3,4) circle (1.5pt);
\draw [color=black] (3,4) circle (1.5pt);
\fill [color=white] (3,3) circle (1.5pt);
\draw [color=black] (3,3) circle (1.5pt);
\fill [color=white] (3,2) circle (1.5pt);
\draw [color=black] (3,2) circle (1.5pt);
\fill [color=white] (3,1) circle (1.5pt);
\draw [color=black] (3,1) circle (1.5pt);
\end{scriptsize}
\end{tikzpicture}\label{graphr}}  
\subfigure[$R$ case (a)]{\begin{tikzpicture}[line cap=round,line join=round,>=triangle 45,x=1.0cm,y=1.0cm]
\clip(0.5,0.5) rectangle (3.5,4.5);
\draw (1,4)-- (2,4);
\draw (2,1)-- (1,1);
\draw (1,4)-- (1,1);
\draw [dash pattern=on 1pt off 1pt] (2,4)-- (3,4);
\draw [dash pattern=on 1pt off 1pt] (3,3)-- (2,4);
\draw [dash pattern=on 1pt off 1pt] (3,4)-- (3,3);
\draw [dash pattern=on 1pt off 1pt] (3,3)-- (3,2);
\draw [dash pattern=on 1pt off 1pt] (3,2)-- (2,1);
\draw [dash pattern=on 1pt off 1pt] (3,1)-- (3,2);
\draw [dash pattern=on 1pt off 1pt] (3,1)-- (2,1);
\begin{scriptsize}
\fill [color=black] (1,4) circle (1.5pt);
\fill [color=black] (1,1) circle (1.5pt);
\fill [color=black] (2,4) circle (1.5pt);
\fill [color=black] (2,1) circle (1.5pt);
\fill [color=white] (3,4) circle (1.5pt);
\draw [color=black] (3,4) circle (1.5pt);
\fill [color=white] (3,3) circle (1.5pt);
\draw [color=black] (3,3) circle (1.5pt);
\fill [color=white] (3,2) circle (1.5pt);
\draw [color=black] (3,2) circle (1.5pt);
\fill [color=white] (3,1) circle (1.5pt);
\draw [color=black] (3,1) circle (1.5pt);
\end{scriptsize}
\end{tikzpicture}\label{casea}}
\subfigure[$R$ case (b)]{\begin{tikzpicture}[line cap=round,line join=round,>=triangle 45,x=1.0cm,y=1.0cm]
\clip(0.5,0.5) rectangle (3.5,4.5);
\draw (1,4)-- (2,3);
\draw (2,2)-- (1,1);
\draw (1,4)-- (1,1);
\draw [dash pattern=on 1pt off 1pt] (2,3)-- (3,4);
\draw [dash pattern=on 1pt off 1pt] (3,3)-- (2,3);
\draw [dash pattern=on 1pt off 1pt] (3,4)-- (3,3);
\draw [dash pattern=on 1pt off 1pt] (3,3)-- (3,2);
\draw [dash pattern=on 1pt off 1pt] (3,2)-- (2,2);
\draw [dash pattern=on 1pt off 1pt] (3,1)-- (3,2);
\draw [dash pattern=on 1pt off 1pt] (3,1)-- (2,2);
\draw [dash pattern=on 1pt off 1pt] (2,3)-- (3,2);
\draw [dash pattern=on 1pt off 1pt] (3,3)-- (2,2);
\draw (2,3)-- (2,2);
\begin{scriptsize}
\fill [color=black] (1,4) circle (1.5pt);
\fill [color=black] (1,1) circle (1.5pt);
\fill [color=black] (2,3) circle (1.5pt);
\fill [color=black] (2,2) circle (1.5pt);
\fill [color=white] (3,4) circle (1.5pt);
\draw [color=black] (3,4) circle (1.5pt);
\fill [color=white] (3,3) circle (1.5pt);
\draw [color=black] (3,3) circle (1.5pt);
\fill [color=white] (3,2) circle (1.5pt);
\draw [color=black] (3,2) circle (1.5pt);
\fill [color=white] (3,1) circle (1.5pt);
\draw [color=black] (3,1) circle (1.5pt);
\end{scriptsize}
\end{tikzpicture}\label{caseb}}
\subfigure[$R$ case (c)]{\begin{tikzpicture}[line cap=round,line join=round,>=triangle 45,x=1.0cm,y=1.0cm]
\clip(0.5,0.5) rectangle (3.5,4.5);
\draw (1,4)-- (2,3);
\draw (2,1)-- (1,1);
\draw (1,4)-- (1,1);
\draw [dash pattern=on 1pt off 1pt] (2,3)-- (3,4);
\draw [dash pattern=on 1pt off 1pt] (3,3)-- (2,3);
\draw [dash pattern=on 1pt off 1pt] (3,4)-- (3,3);
\draw [dash pattern=on 1pt off 1pt] (3,3)-- (3,2);
\draw [dash pattern=on 1pt off 1pt] (3,2)-- (2,1);
\draw [dash pattern=on 1pt off 1pt] (3,1)-- (3,2);
\draw [dash pattern=on 1pt off 1pt] (3,1)-- (2,1);
\draw [dash pattern=on 1pt off 1pt] (2,3)-- (3,2);
\begin{scriptsize}
\fill [color=black] (1,4) circle (1.5pt);
\fill [color=black] (1,1) circle (1.5pt);
\fill [color=black] (2,3) circle (1.5pt);
\fill [color=black] (2,1) circle (1.5pt);
\fill [color=white] (3,4) circle (1.5pt);
\draw [color=black] (3,4) circle (1.5pt);
\fill [color=white] (3,3) circle (1.5pt);
\draw [color=black] (3,3) circle (1.5pt);
\fill [color=white] (3,2) circle (1.5pt);
\draw [color=black] (3,2) circle (1.5pt);
\fill [color=white] (3,1) circle (1.5pt);
\draw [color=black] (3,1) circle (1.5pt);
\end{scriptsize}
\end{tikzpicture}\label{casec}}
\subfigure[Case (c1)]{\begin{tikzpicture}[line cap=round,line join=round,>=triangle 45,x=1.0cm,y=1.0cm]
\clip(1.5,0.5) rectangle (4.5,4.5);
\draw (2,4)-- (2,1);
\draw (2,1)-- (3,1);
\draw (3,3)-- (2,4);
\draw (4,3)-- (3,3);
\draw (3,1)-- (4,1);
\begin{scriptsize}
\fill [color=black] (2,4) circle (1.5pt);
\fill [color=black] (2,1) circle (1.5pt);
\fill [color=black] (3,3) circle (1.5pt);
\fill [color=black] (3,1) circle (1.5pt);
\fill [color=black] (4,3) circle (1.5pt);
\fill [color=black] (4,1) circle (1.5pt);
\end{scriptsize}
\end{tikzpicture}\label{casec1}}
\subfigure[Case (c2)]{\begin{tikzpicture}[line cap=round,line join=round,>=triangle 45,x=1.0cm,y=1.0cm]
\clip(1.5,0.5) rectangle (4.5,4.5);
\draw (2,4)-- (2,1);
\draw (2,1)-- (3,1);
\draw (3,3)-- (2,4);
\draw (4,4)-- (3,3);
\draw (3,1)-- (4,1);
\begin{scriptsize}
\fill [color=black] (2,4) circle (1.5pt);
\fill [color=black] (2,1) circle (1.5pt);
\fill [color=black] (3,3) circle (1.5pt);
\fill [color=black] (3,1) circle (1.5pt);
\fill [color=black] (4,4) circle (1.5pt);
\fill [color=black] (4,1) circle (1.5pt);
\end{scriptsize}
\end{tikzpicture}\label{casec2}}
\subfigure[Case (c3)]{\begin{tikzpicture}[line cap=round,line join=round,>=triangle 45,x=1.0cm,y=1.0cm]
\clip(1.5,0.5) rectangle (4.5,4.5);
\draw (2,4)-- (2,1);
\draw (2,1)-- (3,1);
\draw (3,3)-- (2,4);
\draw (4,3)-- (3,3);
\draw (3,1)-- (4,2);
\draw (3,3)-- (4,2);
\draw (4,2)-- (4,3);
\begin{scriptsize}
\fill [color=black] (2,4) circle (1.5pt);
\fill [color=black] (2,1) circle (1.5pt);
\fill [color=black] (3,3) circle (1.5pt);
\fill [color=black] (3,1) circle (1.5pt);
\fill [color=black] (4,3) circle (1.5pt);
\fill [color=black] (4,2) circle (1.5pt);
\end{scriptsize}
\end{tikzpicture}\label{casec3}}
\subfigure[Case (c4)]{\begin{tikzpicture}[line cap=round,line join=round,>=triangle 45,x=1.0cm,y=1.0cm]
\clip(1.5,0.5) rectangle (4.5,4.5);
\draw (2,4)-- (2,1);
\draw (2,1)-- (3,1);
\draw (3,3)-- (2,4);
\draw (4,4)-- (3,3);
\draw (3,1)-- (4,2);
\draw (3,3)-- (4,2);
\begin{scriptsize}
\fill [color=black] (2,4) circle (1.5pt);
\fill [color=black] (2,1) circle (1.5pt);
\fill [color=black] (3,3) circle (1.5pt);
\fill [color=black] (3,1) circle (1.5pt);
\fill [color=black] (4,4) circle (1.5pt);
\fill [color=black] (4,2) circle (1.5pt);
\end{scriptsize}
\end{tikzpicture}\label{casec4}}
\caption{Auxiliary graphs used in Theorem \ref{theosubgrapdstrongprod} (II).}
\end{figure}

\textbf{Case (iii):} As we have done before, we can assume that only one vertex of $(5,1)$ and $(5,2)$ is in $W$, the same for $(5,3)$ and $(5,4)$. There are four possible cases shown in Figures 
\ref{caseiiia}, \ref{caseiiib}, \ref{caseiiic} and \ref{caseiiid}. In any case, by Theorem \ref{edgesubd}, we can add the edge $\{(4,1),(4,3)\}$ to obtain a graph $T$ such that 
$I(W)\simeq\Sigma I(T)$. In the first case we have that $N_T((4,1))\subseteq N_T((5,3))$, thus $I(H)\simeq I(T-(5,3))$ by Lemma \ref{lemmvert}. But $T-(5,3)$ is an induced subgraph of a graph isomorphic to 
$P_{n-3}\boxtimes P_4$, therefore this case is not possible. In the next two case we have that $T$ is an induced subgraph of a graph isomorphic to $P_{n-3}\boxtimes P_4$, therefore these cases are not 
possible. Thus the only possible case is Figure \ref{caseiiid}.
If $n=5$, then $T\cong P_4$ and $I(T)\simeq*$. Then $n\geq6$. By Theorem \ref{theostrclstrdgr2}, $I(T)\simeq\Sigma I(W')$, with $W'$a graph like the graph in Figure \ref{graphw}.
If $n=6$, then $I(W')$ is homotopy equivalent to $\mathbb{S}^0$ or to $\mathbb{S}^0\vee\mathbb{S}^0$. Thus $n\geq7$. We can proceed as at the beginning of case (iii) until we are again in the case of Figure \ref{caseiiid}.
If $n=7$, then $W'\cong P_5$ and $I(W')\simeq\mathbb{S}^1$.Therefore $n\geq8$ and, as before, by Theorems \ref{edgesubd} and 
\ref{theostrclstrdgr2} we get a graph similar to the graph in Figure \ref{graphw}. We can keep doing this, but the graph is finite, thus this case is not possible. 

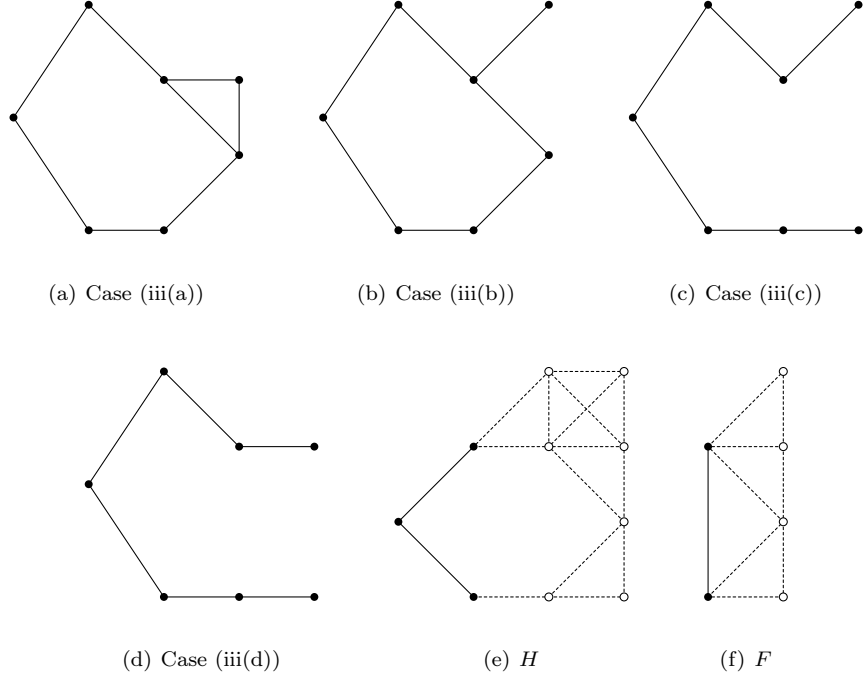
\begin{figure}
\centering
\subfigure[Case (iii(a))]{\begin{tikzpicture}[line cap=round,line join=round,>=triangle 45,x=1.0cm,y=1.0cm]
\clip(0.5,0.5) rectangle (4.5,4.5);
\draw (2,1)-- (3,1);
\draw (3,3)-- (2,4);
\draw (3,1)-- (4,2);
\draw (3,3)-- (4,2);
\draw (1,2.5)-- (2,4);
\draw (1,2.5)-- (2,1);
\draw (4,3)-- (3,3);
\draw (4,3)-- (4,2);
\begin{scriptsize}
\fill [color=black] (2,4) circle (1.5pt);
\fill [color=black] (2,1) circle (1.5pt);
\fill [color=black] (3,3) circle (1.5pt);
\fill [color=black] (3,1) circle (1.5pt);
\fill [color=black] (4,2) circle (1.5pt);
\fill [color=black] (4,3) circle (1.5pt);
\fill [color=black] (1,2.5) circle (1.5pt);
\end{scriptsize}
\end{tikzpicture}\label{caseiiia}}
\subfigure[Case (iii(b))]{\begin{tikzpicture}[line cap=round,line join=round,>=triangle 45,x=1.0cm,y=1.0cm]
\clip(0.5,0.5) rectangle (4.5,4.5);
\draw (2,1)-- (3,1);
\draw (3,3)-- (2,4);
\draw (3,1)-- (4,2);
\draw (3,3)-- (4,2);
\draw (1,2.5)-- (2,4);
\draw (1,2.5)-- (2,1);
\draw (4,4)-- (3,3);
\begin{scriptsize}
\fill [color=black] (2,4) circle (1.5pt);
\fill [color=black] (2,1) circle (1.5pt);
\fill [color=black] (3,3) circle (1.5pt);
\fill [color=black] (3,1) circle (1.5pt);
\fill [color=black] (4,2) circle (1.5pt);
\fill [color=black] (4,4) circle (1.5pt);
\fill [color=black] (1,2.5) circle (1.5pt);
\end{scriptsize}
\end{tikzpicture}\label{caseiiib}}
\subfigure[Case (iii(c))]{\begin{tikzpicture}[line cap=round,line join=round,>=triangle 45,x=1.0cm,y=1.0cm]
\clip(0.5,0.5) rectangle (4.5,4.5);
\draw (2,1)-- (3,1);
\draw (3,3)-- (2,4);
\draw (3,1)-- (4,1);
\draw (1,2.5)-- (2,4);
\draw (1,2.5)-- (2,1);
\draw (4,4)-- (3,3);
\begin{scriptsize}
\fill [color=black] (2,4) circle (1.5pt);
\fill [color=black] (2,1) circle (1.5pt);
\fill [color=black] (3,3) circle (1.5pt);
\fill [color=black] (3,1) circle (1.5pt);
\fill [color=black] (4,1) circle (1.5pt);
\fill [color=black] (4,4) circle (1.5pt);
\fill [color=black] (1,2.5) circle (1.5pt);
\end{scriptsize}
\end{tikzpicture}\label{caseiiic}}
\subfigure[Case (iii(d))]{\begin{tikzpicture}[line cap=round,line join=round,>=triangle 45,x=1.0cm,y=1.0cm]
\clip(0.5,0.5) rectangle (4.5,4.5);
\draw (2,1)-- (3,1);
\draw (3,3)-- (2,4);
\draw (3,1)-- (4,1);
\draw (1,2.5)-- (2,4);
\draw (1,2.5)-- (2,1);
\draw (4,3)-- (3,3);
\begin{scriptsize}
\fill [color=black] (2,4) circle (1.5pt);
\fill [color=black] (2,1) circle (1.5pt);
\fill [color=black] (3,3) circle (1.5pt);
\fill [color=black] (3,1) circle (1.5pt);
\fill [color=black] (4,1) circle (1.5pt);
\fill [color=black] (4,3) circle (1.5pt);
\fill [color=black] (1,2.5) circle (1.5pt);
\end{scriptsize}
\end{tikzpicture}\label{caseiiid}}
\subfigure[$H$]{\begin{tikzpicture}[line cap=round,line join=round,>=triangle 45,x=1.0cm,y=1.0cm]
\clip(0.5,0.5) rectangle (4.5,4.5);
\draw (1,2)-- (2,3);
\draw (2,1)-- (1,2);
\draw [dash pattern=on 1pt off 1pt] (3,4)-- (2,3);
\draw [dash pattern=on 1pt off 1pt] (2,3)-- (3,3);
\draw [dash pattern=on 1pt off 1pt] (3,3)-- (3,4);
\draw [dash pattern=on 1pt off 1pt] (3,4)-- (4,4);
\draw [dash pattern=on 1pt off 1pt] (4,3)-- (3,3);
\draw [dash pattern=on 1pt off 1pt] (4,3)-- (4,4);
\draw [dash pattern=on 1pt off 1pt] (4,4)-- (3,3);
\draw [dash pattern=on 1pt off 1pt] (4,3)-- (3,4);
\draw [dash pattern=on 1pt off 1pt] (4,2)-- (4,3);
\draw [dash pattern=on 1pt off 1pt] (4,2)-- (3,3);
\draw [dash pattern=on 1pt off 1pt] (3,1)-- (4,1);
\draw [dash pattern=on 1pt off 1pt] (4,1)-- (4,2);
\draw [dash pattern=on 1pt off 1pt] (3,1)-- (4,2);
\draw [dash pattern=on 1pt off 1pt] (2,1)-- (3,1);
\begin{scriptsize}
\fill [color=black] (1,2) circle (1.5pt);
\fill [color=black] (2,3) circle (1.5pt);
\fill [color=black] (2,1) circle (1.5pt);
\fill [color=white] (3,4) circle (1.5pt);
\draw [color=black] (3,4) circle (1.5pt);
\fill [color=white] (3,3) circle (1.5pt);
\draw [color=black] (3,3) circle (1.5pt);
\fill [color=white] (3,1) circle (1.5pt);
\draw [color=black] (3,1) circle (1.5pt);
\fill [color=white] (4,1) circle (1.5pt);
\draw [color=black] (4,1) circle (1.5pt);
\fill [color=white] (4,2) circle (1.5pt);
\draw [color=black] (4,2) circle (1.5pt);
\fill [color=white] (4,3) circle (1.5pt);
\draw [color=black] (4,3) circle (1.5pt);
\fill [color=white] (4,4) circle (1.5pt);
\draw [color=black] (4,4) circle (1.5pt);
\end{scriptsize}
\end{tikzpicture}\label{graphh2}}
\subfigure[$F$]{\begin{tikzpicture}[line cap=round,line join=round,>=triangle 45,x=1.0cm,y=1.0cm]
\clip(0.5,0.5) rectangle (2.5,4.5);
\draw (1,3)-- (1,1);
\draw [dash pattern=on 1pt off 1pt] (1,3)-- (2,4);
\draw [dash pattern=on 1pt off 1pt] (2,3)-- (1,3);
\draw [dash pattern=on 1pt off 1pt] (2,3)-- (2,4);
\draw [dash pattern=on 1pt off 1pt] (2,2)-- (1,3);
\draw [dash pattern=on 1pt off 1pt] (2,3)-- (2,2);
\draw [dash pattern=on 1pt off 1pt] (2,1)-- (1,1);
\draw [dash pattern=on 1pt off 1pt] (1,1)-- (2,2);
\draw [dash pattern=on 1pt off 1pt] (2,2)-- (2,1);
\begin{scriptsize}
\fill [color=black] (1,1) circle (1.5pt);
\fill [color=black] (1,3) circle (1.5pt);
\fill [color=white] (2,4) circle (1.5pt);
\draw [color=black] (2,4) circle (1.5pt);
\fill [color=white] (2,3) circle (1.5pt);
\draw [color=black] (2,3) circle (1.5pt);
\fill [color=white] (2,2) circle (1.5pt);
\draw [color=black] (2,2) circle (1.5pt);
\fill [color=white] (2,1) circle (1.5pt);
\draw [color=black] (2,1) circle (1.5pt);
\end{scriptsize}
\end{tikzpicture}\label{graphf}}
\caption{Auxiliary graphs used in Theorem \ref{theosubgrapdstrongprod} (III).}
\end{figure}

Therefore only one of $(1,2)$ and $(1,3)$ is in $H$. We assume that $(1,2)$ is in $V(H)$ and $(1,3)$ is not. If $(2,2)$ is a vertex of $H$, by Lemma \ref{lemmvert} we have that 
$$I(H)\simeq I(H-(2,2))\vee I(H-N_H[(2,2)])$$
and we are only interested in $H-(2,2)$. So we can assume that $(2,2)$ is not a vertex of $H$. Then $(2,1)$ and $(2,3)$ are in $V(H)$, otherwise $(1,2)$ would have degree $1$ or $0$ and $H$
would not be a counterexample. If $(2,4)$ is a vertex of $H$, then 
$$I(H)\simeq I(H-(2,3))\vee\Sigma I(H-N_H[(2,3)])\;\;\mbox{ and }$$
$$I(H-(2,3))\simeq\Sigma I(H-(2,3)-N_H[(2,1)]).$$
This would imply that $H$ is not a counterexample. Thus  $(2,4)$ is not a vertex of $H$. If $(3,2)$ is in $V(H)$, 
then $I(H)\simeq I(H-(3,2))$  by Lemma \ref{lemmvert} because $N_H((1,2))\subseteq N_H((3,2))$. 
Thus we can assume that $(3,2)$ is not in $V(H)$ (see Figure \ref{graphh2}). If $(3,1)$ is not in $H$, then $d_H((2,1))=1$ and 
$I(H)\simeq\Sigma I(H-N_H[(1,2)])$. This would imply that $H$ is not a counterexample, thus $(3,1)$ is a vertex of $H$. If $d_H((1,3))=1$, then $I(H)\simeq\Sigma I(H-N_H[(1,2)])$ and  
$H$ would not be a counterexample. If $(3,3)$ and $(3,4)$ are in $V(H)$, then 
$$I(H)\simeq I(H-(3,3))\vee\Sigma I(H-N_H[(3,3)])$$
and $I(H-N_H[(3,3)])$ would have the homotopy type of a wedge of spheres or would be contractible. Therefore we can assume that only one of $(3,3)$ and $(3,4)$ is in $V(H)$. If $(3,4)$ is in 
$V(H)$, then we are in the case(ii) with $H$ like the graph in Figure \ref{caseii}. Then $(3,4)$ is not in $V(H)$. Taking $F$ as the graph obtained from $H-N_H[(1,2)]$ by adding the edge 
$\{(3,1),(3,3)\}$ (see Figure \ref{graphf}), we have that $I(H)\simeq\Sigma I(F)$ by Theorem \ref{edgesubd}. As  we have done before, we can assume that only one vertex of 
$(4,3)$ and $(4,4)$ is in $F$. If $(4,2)$ is a vertex of $F$, then $N_F[(3,1)]\subseteq N_F[(4,2)]$ and 
$$I(F)\simeq I(F-(4,2))\vee\Sigma I(F-N_F[(4,2)])$$
where we are only interested in $I(F-(4,2))$. Thus we can assume that $(4,2)$ is not a vertex of $F$. There are two possible cases. If $(4,4)$ is in $V(F)$, then $F$ would be an induced subgraph 
of a graph isomorphic to $P_{n-2}\boxtimes P_4$. Assume that $(4,3)$ is in $V(F)$. If $n=4$, then $F\cong P_4$ and $I(F)\simeq*$. Thus $n\geq5$. By Theorem \ref{theostrclstrdgr2} 
$I(F)\simeq\Sigma I(F')$ where $F'$ is like the graph in Figure \ref{graphw}. Therefore there is no counter example.
\end{proof}

\begin{que}\label{quesubgrp}
Can Theorem \ref{theosubgrapdstrongprod} be extended to any $m$?
\end{que}

\begin{theorem}\label{propmk}
Let $m$ be a positive integer such that for any $n$, 
$P_n\boxtimes P_m$ is in $\mathcal{SW}$, then $I(P_n\boxtimes P_{m+k})$ has the homotopy type of a 
wedge of spheres for $1\leq k\leq 4$.
\end{theorem}
\begin{proof}
We take $G=P_n\boxtimes P_{m+k}$, with $k=1,2$. Then, for all $j$ in $\underline{m+l}$ we have that 
$$N_G[(j,1)]\subseteq N_G[(j,2)]\mbox{ and }N_W[(j,1)]\subseteq N_W[(j,2)]$$
for any induced subgraph $W$ such that $(j,1)$ and $(j,2)$ are in $V(W)$. This last fact together with Lemma \ref{lemmvert} will allow us to 
decompose $I(G)$ into a wedge of spaces we know have the homotopy type of a wedge of spheres. Taking the vertex set 
$S=\{(1,2),(2,2).\dots(n,2)\}$, the idea is to remove in each step a vertex or a vertex and its closed neighborhood. 
As fist step we have that 
$$I(G)\simeq I(G-(1,2))\vee\Sigma I(G-N_G[(1,2)]).$$
Next, we have that 
$$I(G-(1,2))\simeq I(G-(1,2)-(2,2))\vee\Sigma I(G-(1,2)-N_G[(2,2)])$$
and
$$I(G-N_G[(1,2)])\simeq\Sigma I(G-N_G[(1,2)]-(3,2))\vee\Sigma I(G-N_G[(1,2)]-N_G[(3,2)]).$$
We keep doing this until we get a family of complexes $I(W_1),\dots,I(W_p)$
such that none of them have any vertex of $S$.
To each of these complexes we can associate a sequence $s^j=(s_1,s_2,\dots,s_{r_j})$ such that $s_i=1,2$  and $s_1+\cdots+s_{r_j}=n$ or 
$s_1+\cdots+s_{r_j}=n+1$, where: 
\begin{enumerate}
    \item For $1$, $s_1$ tell us if we get the complex form $I(G-(1,2))$ or $I(G-N_G[(1,2)])$.
    \item For $i\geq2$, $s_i$ tell us if in the step $i$ we erased $(s_1+\cdots+s_{i-1}+1,2)$ or its closed neighborhood.
\end{enumerate}
For each $j=1,\dots,p$, we take $t_j$ equal to the number of entries of $s^j$ which are equal to $2$. Therefore
$$I(G)\simeq\bigvee_{j=1}^p\Sigma^{t_j}I(W_j).$$
Now, each $W_j$ is the disjoint union of induced subgraphs of $P_n\boxtimes P_{m}$ and paths-- we are taking $P_0=\emptyset$. 
From this we get the result.

Take $G=P_n\boxtimes P_{m+k}$, with $k=3,4$ and do the same as before. Then 
$$I(G)\simeq\bigvee_{j=1}^p\Sigma^{t_j}I(W_j).$$
Now, we take $R=\{(1,m+k-1),(2,m+k-1).\dots(n,m+k-1)\}$. As before 
$$N_G[(j,m+k)]\subseteq N_G[(j,m+k-1)]\mbox{ and }N_W[(j,m+k)]\subseteq N_W[(j,m+k-1)]$$
for any induced subgraph $W$ such that $(j,m+k)$ and $(j,m+k-1)$ are in $V(W)$. For each $W_j$ we do the same process as before on the 
vertex set $R$. 
\end{proof}

\begin{cor}
For any $n$ and $1\leq m\leq8$, $I(P_n\boxtimes P_m)$ has the homotopy type of a wedge of spheres.
\end{cor}
\begin{proof}
This is a direct consequence of Theorem \ref{theosubgrapdstrongprod} and Proposition \ref{propmk}.
\end{proof}

\begin{que}
Is the homotopy type of $I(P_n\boxtimes P_m)$ always a wedge of spheres for all $n,m$?
\end{que}
A positive answer to Question \ref{quesubgrp} will give us a positive answer to last question by Theorem \ref{theosubgrapdstrongprod}.
\bibliographystyle{acm}
\bibliography{matchcomplex}
\end{document}